%% file: main.tex
\documentclass[11pt]{amsart}

\input{Preamble}

\title{Nerve Models of Subdivision Bifiltrations}
\author{Michael Lesnick and Ken McCabe}
\date{\today}

\begin{document}

\maketitle

\begin{abstract}
We study the size of Sheehy's subdivision bifiltrations, up to homotopy. We focus in particular on the subdivision-Rips bifiltration $\srips(X)$ of a metric space $X$,  the only density-sensitive bifiltration on metric spaces known to satisfy a strong robustness property.  Given a simplicial filtration $\mathcal F$ with a total of $m$ maximal simplices across all indices, we introduce a nerve-based simplicial model for its subdivision bifiltration $\subdiv$ whose $k$-skeleton has size $O(m^{k+1})$.  We also show that the $0$-skeleton of any simplicial model of $\subdiv$ has size at least $m$.   We give several applications:  For an arbitrary metric space $X$, we introduce a $\sqrt{2}$-approximation to $\srips(X)$, denoted $\mathcal J(X)$, whose $k$-skeleton has size $O(|X|^{k+2})$.  This improves on the previous best approximation bound of $\sqrt{3}$, achieved by the degree-Rips bifiltration, which implies that $\mathcal J(X)$ is more robust than degree-Rips.  Moreover, we show that the approximation factor of $\sqrt{2}$ is tight; in particular, there exists no exact model of $\srips(X)$ with poly-size skeleta.  On the other hand, we show that for $X$ in a fixed-dimensional Euclidean space with the $\ell_p$-metric, there exists an exact model of $\srips(X)$ with poly-size skeleta for $p\in \{1, \infty\}$, as well as a $(1+\eps)$-approximation to $\srips(X)$ with poly-size skeleta for any $p \in (1, \infty)$ and fixed ${\eps > 0}$.
\end{abstract}
\tableofcontents

\section{Introduction}
\label{sec:Introduction}
\subsection{Background}

Topological data analysis (TDA) \cite{chazalIntroductionTopological2021, henselSurveyTopological2021, wassermanTopologicalData2018,carlsson2021topological} studies the shape of a data set by constructing a diagram of topological spaces and computing invariants of this diagram.  In particular, persistent homology \cite{edelsbrunnerPersistentHomology2008, punPersistenthomologybasedMachine2022,ghristBarcodesPersistent2008}, the mostly widely studied and applied TDA technique, begins by constructing a 1-parameter filtration of spaces, e.g., the \emph{\v{C}ech filtration} of a Euclidean point cloud or the \emph{(Vietoris-)Rips filtration} of a finite metric space \cite{edelsbrunnerComputationalTopology2022}.  Applying the homology functor with field coefficients to a filtration yields a diagram of vector spaces called a persistence module, whose isomorphism type is fully described by a multiset of intervals called a \emph{barcode}.

The barcodes of \v Cech and Rips filtrations are well-known to be stable with respect to small perturbations of the data \cite{cohen-steinerStability07,chazal2009gromov,chazal2014persistence}.  However, these barcodes can be highly unstable with respect to outliers, and also insensitive to variations in density \cite[Section ~$4$]{blumbergRobust14}.  Several strategies have been proposed to address these limitations within the framework of 1-parameter persistence; see \cite[Section 1.7]{blumbergStability2Parameter2022}.  However, as noted in \cite{blumbergStability2Parameter2022}, these strategies share a common disadvantage, which is that they depend on a choice of one or more parameters.  It can be unclear how to best choose these parameters, and in some settings, it can be that no single choice allows us to fully capture the topological structure of interest in our data.  This motivates us to analyze the topology of point cloud or metric data by constructing a \emph{bifiltration}, treating density and spatial scale as separate parameters \cite{carlssonTheoryMultidimensional2009}.  

The branch of TDA that studies data via the homology of multiparameter filtrations is known as \emph{multiparameter persistent homology (MPH)} \cite{carlssonTheoryMultidimensional2009, botnanIntroductionMultiparameter2023}.  Though there are difficulties with extending the notion of a barcode to the multiparameter setting, several simple, albeit incomplete, invariants of multiparameter persistence modules have been proposed and studied.  These can be used as surrogates for a barcode in applications; see \cite[Sections 4 and 9]{botnanIntroductionMultiparameter2023} for an overview.

The problem of working with 2-parameter persistent homology in a practical, computationally efficient way has been actively studied in recent years, leading to significant advances in algorithms and software \cite{fugacci2023compression,lesnickComputingMinimal2022a,lesnickInteractiveVisualization2015,bauer2023efficient,alonso2023filtration,alonsoDelaunayBifiltrations2023,carriere2020multiparameter,vipond2018multiparameter,xin2023gril,loiseaux2024stable,rolle2020stable}, as well as several interesting applications \cite{vipond2021multiparameter,demir2022todd,benjamin2022multiscale,chen2021tamp,adcock2014classification,xia2015multidimensional,loiseaux2024stable}.  We expect such applications will become more common as multiparameter persistence software continues to improve and is more fully integrated into standard ecosystems for TDA and machine learning.  However, the applicability of some of the prevailing 2-parameter approaches is limited by relatively high computational cost, and much current research in this area centers on limiting this cost.

\subsection{Density-Sensitive Bifiltrations}
 There are multiple ways of constructing a density-sensitive bifiltration from point cloud or metric data, which vary in their robustness to outliers, computability, and generality \cite{blumbergRobust14}.  We now review several of these constructions, focusing on those that do not depend on an additional parameter choice; see \cite{carlssonTheoryMultidimensional2009, chazalGeometricInference2011, alonsoDelaunayBifiltrations2023, blaser2024core} for other constructions.
 
We denote the Rips filtration of a finite metric space $X$ as $\rips(X)$.  For $X$ equipped with an embedding into an ambient metric space $Y$ (e.g., $Y=\R^n$), we denote the \v Cech filtration of $X$ as $\cech(X)$, suppressing the dependence on $Y$ in the notation; see \cref{Sec:Rips_Cech} for the definitions.
Sheehy \cite{sheehyMulticoverNerve2012a} introduced the  \emph{subdivision-Rips} and \emph{subdivision-\v{C}ech bifiltrations} $\srips(X)$ and $\scech(X)$, which refine $\mathcal R(X)$ and $\cech(X)$, respectively.   In fact, Sheehy's construction provides a bifiltered refinement $\mathcal S \mathcal F$ of any simplicial filtration $\mathcal F$.   The definition of $\mathcal{SF}$ amounts to the observation that there is a natural filtration on the barycentric subdivision of any simplicial complex.  Unfortunately, the sizes of both $\srips(X)$ and $\scech(X)$ are exponential in $|X|$, so in practice we cannot directly compute them.  
 
Sheehy, Cavanna, and Gardners' \emph{multicover nerve theorem} \cite{cavanna17when,sheehyMulticoverNerve2012a,blumbergStability2Parameter2022}, a 2-parameter extension of the standard topological equivalence between the \v Cech and union-of-balls filtrations, motivates and illuminates the subdivision-\v Cech bifiltration.  The theorem says that if finite intersections of balls are contractible in the ambient metric space $Y$, then $\scech(X)$ is homotopically equivalent (i.e., \emph{weakly equivalent}, see \Cref{sec:Filtrations}) to the \emph{multicover bifiltration} $\mathcal M(X)$.  The bifiltration $\mathcal M(X)$ is an extension of the union-of-balls filtration, where one filters both by radius and by \emph{cover multiplicity}, i.e., the minimum number of balls covering a point.

In the case that $X\subset \R^d$ is in general position, Corbet et al. \cite{corbetComputingMulticover2023} showed that the \emph{rhomboid bifiltration} $\rhomb(X)$, a polyhedral bifiltration of size $\Theta(|X|^{d+1})$ introduced by Edelsbrunner and Osang \cite{edelsbrunnerMultiCoverPersistence2021}, is homotopically equivalent to $\mathcal M(X)$, and hence also to $\scech(X)$.  This improves on earlier results from \cite{edelsbrunnerMultiCoverPersistence2021} concerning 1-parameter slices of $\mathcal M(X)$.  Furthermore, $\rhomb(X)$ can be computed in polynomial time \cite{edelsbrunnerSimpleAlgorithm2020}, and an implementation is available for the cases $d=2,3$ \cite{osangRhomboidTiling2020}.  Even so, the size of $\rhomb(X)$ makes its full computation impractical for $d=3$.  
Recent work by Buchet et al.\ \cite{buchetSparseHigher2023} gives a linear-size approximation of the restriction of $\mathcal M(X)$ to cover multiplicities at most some constant $\mu$.    However, the size of this construction depends exponentially on $\mu$.

Lesnick and Wright \cite{lesnickInteractiveVisualization2015} introduced another simple density-sensitive bifiltration on a finite metric space $X$, the \emph{degree-Rips bifiltration} $\mathcal{DR}(X)$, whose $k$-skeleton has size at most $O(|X|^{k+2})$.  Restricting $\mathcal{DR}(X)$ to a grid of constant size yields a bifiltration whose skeleta have size $O(|X|^{k+1})$, which asymptotically matches the size of the 1-parameter Rips filtration.  This is small enough that $H_0(\mathcal{DR}(X))$ and invariants thereof can readily be computed for real data sets with thousands of points.  The software packages RIVET \cite{rivet} and Persistable \cite{Scoccola2023} support such computations.

Furthermore, Blumberg and Lesnick \cite{blumbergStability2Parameter2022} showed that $\scech(-)$ and $\srips(-)$ both satisfy strong robustness properties, closely paralleling the standard stability theorems of single-parameter persistence.  
 In the case of $\srips(-)$, the result is the following:
\begin{theorem}[{\cite[Theorem 1.6\,(iii)]{blumbergStability2Parameter2022}}] \label{Thm:Robustness}
For any finite, non-empty metric spaces $X$ and $X'$, the homotopy interleaving distance between $\srips(X)$ and $\srips(X')$ is at most the Gromov-Prohorov distance between the uniform probability measures on $X$ and $X'$. 
\end{theorem}
In addition,  \cite[Theorem ~$1.7$\,(ii)]{blumbergStability2Parameter2022} gives a robustness result for $\mathcal{DR}(-)$ analogous to \cref{Thm:Robustness}, formulated in terms of \emph{affine homotopy interleavings}, where the shifts between indices in the interleaving are affine maps rather than translations.  This a far weaker robustness property than that of $\srips(-)$, but is tight, in a sense.  

Rolle and Scoccola \cite{scoccolaLocally20} give a Lipschitz stability result for $\mathcal{DR}(-)$ with respect to the \emph{Gromov-Hausdorff-Prohorov distance}. However, since that distance is not robust to outliers, Rolle and Scoccola's result does not speak to the robustness of $\mathcal{DR}(-)$.

\subsection{Contributions}\label{Sec:Statements}
Together, the above results reveal a tension between between robustness, computability, and generality for density-sensitive bifiltrations: The more general and more robust a bifiltration is, the larger and less computable it is, at least with current technology.  Ideally, one would like a bifiltration that is robust (say, in the sense of \cref{Thm:Robustness}), computable, and defined for arbitrary finite metric spaces.  No such construction is known at present, and the problem of understanding what can and cannot be done in this direction is largely open.  

In this paper, we approach this problem by studying the size (up to homotopy) of $\srips(-)$, the only density-sensitive bifiltration on metric spaces known to satisfy a strong robustness property.  Specifically, we ask if and when there exists a simplicial construction weakly equivalent to $\srips(-)$ with a \emph{poly-size} (i.e., polynomial-size) $k$-skeleton for each fixed $k$.  We consider both exact and approximate formulations of this question.  For context, note that the multicover nerve theorem and the main result of \cite{corbetComputingMulticover2023} together imply that for data in $\R^d$ for $d$ fixed, there exists a poly-size cellular bifiltration weakly equivalent to $\scech(-)$, namely the rhomboid bifiltration $\mathcal T(-)$; our aim here is to develop analogous results for $\srips(-)$.  

In what follows, we state our main results, deferring some key definitions to later sections.   

Let $\Top$ and $\Simp$ denote the categories of topological spaces and finite simplicial complexes, respectively.  We regard $\Simp$ as a subcategory of $\Top$ via geometric realization.

The next definition uses the notion of weak equivalence of $\Top$-valued functors, which we review in \cref{Sec:Weak_Equiv}.

\begin{definition}
Given a poset $P$, a \emph{(simplicial) model} of a functor $\mathcal F\colon \cat{P}\to \Top$ is a functor $\mathcal F'\colon \cat{P}\to \Simp$ such that $\mathcal F$ and $\mathcal F'$ are weakly equivalent.
\end{definition}

\begin{notation}\label{Notation:m_k}
Let $T$ be a totally ordered set, e.g., $T=[0,\infty)$.  For $\mathcal F\colon T\to \Simp$ a filtration and $k\geq 0$, let $m_k=m_k(\mathcal F)$ denote the number of sets $S$ of simplices in  $\colim \mathcal F =\bigcup_{t\in T} \mathcal F_t$ such that 
\begin{itemize}
\item $|S|\leq (k+1)$ and
\item  for some $t\in T$, each $\sigma\in S$ is a maximal simplex (i.e., has no cofacet) in $\mathcal F_t$.
\end{itemize}
\end{notation}

Note that $m_0$ is simply the number of simplices in $\colim \mathcal F$ that are maximal in some $\mathcal F_t$.  Therefore,  $m_k=O(m_0^{k+1})$.

In \cref{sec:nerve-model}, we construct a certain functor $\mathcal{NF}\colon \N^{\op}\times T\to \Simp$ from a simplicial filtration $\mathcal F\colon T\to \Simp $, and prove the following, which is our main result about general subdivision bifiltrations:

\begin{theorem}\label{thm:main-theorem-intro}\label{Cor:Thm_Generalization}\mbox{}
For any finitely presented filtration $\mathcal F\colon T\to \Simp$,
    \begin{itemize}
        \item [(i)] $\mathcal{NF}$ is a simplicial model of $\subdiv$ whose $k$-skeleton has size $O(m_k)$,
        \item [(ii)] the $0$-skeleton of any simplicial model of $\subdiv$ has size at least $m_0$.
    \end{itemize}
\end{theorem}
See \cref{Sec:Size} for the definition a finitely presented filtration and \cref{Def_Size_Filtration} for the definition of size used here.  

In brief, to construct $\mathcal{NF}$, we cover each simplicial complex $\mathcal F_t$ of $\mathcal F$ by its maximal simplices.  These covers induce a \emph{functorial cover} of $\subdiv$; we define $\mathcal{NF}$ to be the nerve of this functorial cover.  To show that $\mathcal{NF}$ is a model of $\subdiv$, we use a functorial version of the nerve theorem \cite{bauerUnifiedView2023}.  See \cref{sec:nerve-model,subsec:nerve-model} for the details.  

We give several applications of \Cref{thm:main-theorem-intro}.  
We use a notion of \emph{$\delta$-approximation} of bifiltrations, whose definition, given in \cref{sec:Interleavings}, is formulated in terms of \emph{(multiplicative) homotopy interleavings} \cite{blumberg2023universality}.  

\begin{corollary}\label{cor:intro-3}~
    \begin{itemize}
 \item[(i)] \begin{sloppypar} For any finite metric space $X$, there exists a simplicial $\sqrt{2}$-approximation $\mathcal J(X)$ to $\srips(X)$ whose $k$-skeleton has size $O(|X|^{k+2})$. \end{sloppypar} 
 \item[(ii)]  There exists an infinite family of finite metric spaces $X$ such that for any $\delta\in [1,\sqrt{2})$, there is no simplicial $\delta$-approximation to $\srips(X)$ with size polynomial in $|X|$.
     \end{itemize}
\end{corollary}

\begin{remark}\label{rem:Robustness}~
\begin{itemize}
\item[(i)] We define $\mathcal J(X)$ as a linear rescaling of $\mathcal {N I}(X)$, where $\mathcal I(X)$ is the \emph{intrinsic \v Cech filtration} of $X$ \cite[Section 4.2.2]{chazal2014persistence},
\item[(ii)] Hellmer and Spaliński \cite{hellmerDensitySensitive2024} have independently introduced a bifiltration which is weakly equivalent to $\mathcal {N I}(X)$ and has the same worst-case asymptotic size. 
\item[(iii)] It follows from \cite[Proposition ~$3.4$]{blumbergStability2Parameter2022} that a linear rescaling of $\mathcal{DR}(X)$ is a $\sqrt{3}$-approximation of $\srips(X)$.  \Cref{cor:intro-3} thus gives an improved (optimal) approximation factor of $\sqrt{2}$, via a construction of the same asymptotic size.  Using this approximation bound, arguments in \cite{blumbergStability2Parameter2022} extend immediately to show that $\mathcal J(X)$ satisfies the same robustness property as the one given for $\mathcal{DR}(X)$ in \cite[Theorem ~$1.7$\,(ii)]{blumbergStability2Parameter2022}, but with an improved constant of 2 replacing the constant 3 in that result.  Hence, as the constant 3 in the robustness result for $\mathcal{DR}(X)$ is tight \cite[Proposition 3.7]{blumbergStability2Parameter2022}, $\mathcal J(X)$ is more robust to outliers than $\mathcal{DR}(X)$.
\end{itemize}
\end{remark}

When $\delta=1$,  \cref{cor:intro-3}\,(ii) says that there exists no poly-sized model of $\srips(X)$ for arbitrary finite metric spaces. In fact, this is true even for $X\subset \R^2$, with the usual $\ell_2$ metric: 

\begin{corollary}\label{cor:intro-plane}\mbox{}
\begin{itemize}
    \item[(i)] For $X$ a finite set of uniformly spaced points on the unit circle, any simplicial model of $\srips(X)$ has exponential size in $|X|$.
    \item[(ii)] Let $X$  be a finite i.i.d.\ sample of the uniform distribution on the unit square.   There exists a constant $C > 0$ such that
 for any simplicial model $\mathcal F(X)$ of $\srips(X)$,    \[
    \operatorname{Pr}[\textup{size of } \mathcal F(X) \textup{ is at least }\exp(C|X|^{1/3})] \rightarrow 1
    \]
    as $|X| \rightarrow \infty$.
\end{itemize}
\end{corollary}

Informally, \cref{cor:intro-plane}\,(ii) tells us that for random data sets $X\subset [0,1]^2$, the size of any model of $\srips(X)$ is exponential in $|X|$ with high probability.  The proof uses a bound on the number of maximal cliques in random geometric graphs, due to Yamaji \cite{yamajiNumberMaximal2023a}.

For $p\in [1,\infty]$, let $\R^d_p$ denote $d$-dimensional Euclidean space, endowed with the $\ell_p$ metric. 

\begin{corollary}\label{cor:intro-2}
 Fixing $d\in \N$, let $X\subset \R^d_p$.
    \begin{itemize}
        \item [(i)] If $d=1$ or $p\in \{1, \infty\}$, then $\mathcal{NR}(X)$ is a simplicial model of $\srips(X)$ whose $k$-skeleton has size $O(|X|^{\alpha(k+1)+2})$, where 
        \[\alpha=\alpha(d,p)=\begin{cases}
        d &\text{ for  $p=\infty$ or $d=1$ or $(d,p)=(2,1)$},\\
        d2^{d^2-1} &\text{ for $p=1$ and $d>2$}. 
        \end{cases}\]
        \item [(ii)] For any $\eps \in (0,1)$, there exists a $(1+\eps)$-approximation to $\srips(X)$ whose $k$-skeleton has size $O(|X|^{\beta (k+1)+2}),$ where 
        $\beta = c_d \eps ^{-(d^2-1)/2}$ for some constant $c_d$ depending $d$.  
        \end{itemize}
\end{corollary}
The main takeaway here is that for fixed $d$,  $\epsilon$, and $k$, the $k$-skeleta considered in  \cref{cor:intro-2} have size polynomial in $|X|$.   However, the quantities $\alpha(d,1)$ and $\beta$ appearing in the exponents of the bounds are large and grow very quickly with $d$.  It might be possible to improve these bounds.

We conjecture that  \Cref{cor:intro-2} (ii) extends to finite metric spaces of bounded doubling dimension \cite[Section $1.3$]{guptaBoundedGeometries2003}, as follows:

\begin{conjecture}
    For $X$ a finite metric space of bounded doubling dimension and fixed $\eps > 0$, there exists a $(1+\eps)$-approximation to $\srips(X)$ whose fixed-dimensional skeleta have size polynomial in $|X|$.
\end{conjecture}

We remark that the functor $\mathcal{NF}$ of \cref{thm:main-theorem-intro}\,(i), which underlies all of our main results, is generally not a bifiltration, but is rather a \emph{semifiltration}, meaning that the structure maps in only one of the two parameter directions are required to be inclusions.  In \cref{Semi_Filtration_to_Bifilration}, we observe that by extending a construction of Kerber and Schreiber \cite{kerber2019barcodes} in the 1-parameter setting, we can convert a simplicial semifiltration of bounded pointwise dimension to a bifiltration, with only logarithmic increase in size.  One can use this construction to give variants of \cref{thm:main-theorem-intro}\,(i), \cref{cor:intro-3}\,(i), and \cref{cor:intro-2} where the functors are bifiltrations.  For example, we have the following corollary of \cref{thm:main-theorem-intro}\,(i):

\begin{corollary}
For a finitely presented filtration $\mathcal F\colon \N\to \Simp$ and fixed $k\geq 0$, there exists a \emph{bifiltered} simplicial model of the $k$-skeleton of $\subdiv$ whose size is $O(m_k \log m_0)$.
\end{corollary}

The fixed-dimensional skeleta of all of our constructions can be computed in time polynomial in the total size of the input and output.  We briefly explain this in \cref{Sec:Computation}, deferring a more thorough study of computation to future work.

\subsection*{Outline}
\cref{sec:Background} covers preliminaries, including filtrations, persistence modules, homotopy interleavings, and the persistent nerve theorem.  \cref{sec:nerve-model} introduces our nerve-based model of the subdivision bifiltration and uses this to prove \cref{thm:main-theorem-intro}.  \cref{sec:intrinsic-cech} gives the proof of \cref{cor:intro-3}, while \cref{sec:taxicab} gives the proofs of \cref{cor:intro-plane,cor:intro-2}.  \cref{Sec:Computation} discusses the polynomial-time computation of our constructions.  \cref{Semi_Filtration_to_Bifilration} explains how the construction of Kerber and Schreiber \cite{kerber2019barcodes} extends to convert a semifiltration to a bifiltration.

\section{Preliminaries}
\label{sec:Background}

We assume familiarity with basic algebraic topology, graph theory, and category theory, as introduced, e.g., in \cite{hatcherAlgebraicTopology2002,bondyGraph08,riehl2017category}.  
For additional background on TDA and persistent homology, see, e.g., \cite{edelsbrunnerComputationalTopology2022,deyComputational22,carlsson2021topological,carlsson2009topology,boissonnat2018geometric,oudotPersistenceTheory2015}.

Let $P$ be a poset, which we regard as a category in the usual way, i.e. the object set is $P$ and morphisms are pairs $p\leq q$.  For any category $\mathbf C$, functor $F:P\to \mathbf C$, and $p\leq q$ in $P$, we write $F(p)$ as $F_p$ and $F(p\leq q)\colon F(p)\to F(q)$ as $F_{p \subto q}$.  We call the morphisms $F_{p \subto q}$ \emph{structure maps}.
The functors $P\to \mathbf C$ form a category $\mathbf C^P$, whose morphisms are the natural transformations.  

Given posets $P_1,\ldots,P_l$, we regard the Cartesian product $P_1\times \cdots\times P_n$ as a poset, where $(p_1,\dots p_n)\leq (q_1,\ldots,q_n)$ if and only if $p_i \leq q_i$ for each $i$.  

\subsection{Filtrations}
\label{sec:Filtrations}
A \emph{($P$-)filtration} is a functor $\mathcal F:\cat{P}\to \Top$ such that the map $\mathcal F_{p\subto q}$ is a subspace inclusion for all $p \leq q$ in $P$. If $\mathcal F$ is $\Simp$-valued, we call $\mathcal F$ a \emph{simplicial} filtration.  When $P$ is a totally ordered set, we call $\mathcal F$ a \emph{1-parameter} filtration.  If $P=T_1\times T_2$ for totally ordered sets $T_1$ and $T_2$, we call $\mathcal F:\cat{P}\to \Top$  a \emph{bifiltration}.

We will also consider functors $\mathcal F:\cat{P}\to \Top$ that are not filtrations, particularly those of the following form:  For $P=T_1\times T_2$ with each $T_i$ totally ordered, we call $\mathcal F: \cat{P}\to \Top$ a \emph{semifiltration} if for each fixed $t \in T_2$, the restriction $\mathcal F_{-, t}: T_1\times \{t\} \to \Top$ is a filtration.

\subsubsection{Clique and \v Cech Filtrations}\label{Sec:Rips_Cech}
A \emph{graph} is a simplicial complex of dimension at most one.  A \emph{graph filtration}  is a graph-valued filtration $\mathcal G\colon P\to \Simp$.  Any such graph filtration  induces a simplicial filtration $\overline{\mathcal G}:P \to \Simp$, called the \emph{clique filtration} of $\mathcal G$, where each $\overline{\mathcal G}_p$ is the \emph{clique complex} of $\mathcal G_p$, i.e., $\overline{\mathcal G}_p$ is the largest simplicial complex with 1-skeleton $\mathcal G_p$.  For a key example, let $X$ be a metric space with metric $d_X$ and let $r \geq 0$.   The \emph{$r$-neighborhood graph} $\mathcal G(X)_r$ is the graph with vertex set $X$ and an edge $[x,y]$ if and only $d_X(x,y) \leq 2r$.  The graphs $\mathcal G (X)_r$ assemble into a graph filtration $\mathcal G (X):\T\to \Simp$. We call $\rips(X)\coloneqq \overline{\mathcal G (X)}$ the \emph{(Vietoris)-Rips filtration of $X$}.

\begin{definition}\label{Def:Nerve}
For a collection of sets $U$, the \emph{nerve of $U$} is the simplicial complex
\[
\Ner(U)\coloneqq \{ S \subset U \mid S \textup{ is finite and }\bigcap_{\sigma \in S} \sigma \neq \emptyset\}.
\]
\end{definition}
Given a metric space $Y$, a point $y\in Y$, and $r \in\T$, let \[B(y)_r=\{z\in Y\mid d_Y(z,y)\leq r\}.\]
For a metric space $X$ equipped with an embedding into a metric space $Y$, let
 \[\cech(X)_r=\Ner(\{B(x)_r \mid x \in X\}),\]
 where each $B(x)_r$ is understood to be a metric ball in $Y$.  Allowing $r$ to vary, we obtain the \emph{\v Cech filtration} \[\cech(X)\colon \T\to \Simp.\]  In TDA, one most often considers \v Cech filtrations where $Y=\R^n$.  However, in this paper we will primarily consider the setting $Y=X$, in which case we call $\cech(X)$ the  \emph{intrinsic \v Cech filtration} of $X$ and denote it as $\mathcal I(X)$.  The stability of  $\mathcal I(X)$ has been studied in \cite[Section 4.2.2]{chazal2014persistence}.
\begin{example}
To illustrate the difference between the intrinsic \v Cech filtration and the usual (extrinsic) construction, let 
\[
X= \{(0,0), (4,0), (2,2\sqrt 3)\} \subset \R^2=Y,
\]
so that $X$ is the vertex set of an equilateral triangle with side length $4$. Write $a,b,c$ for the vertices of $X$, respectively.  Then
   \begin{align*}
    \cech(X)_r &= \begin{cases}
        \{[a],[b],[c]\}, & 0 \leq r < 2, \\
        \{[a],[b],[c],[ab],[ac],[bc]\}, & 2 \leq r < \frac{4\sqrt 3}{3}, \\
        \{[a],[b],[c],[ab],[ac],[bc], [abc]\}, & \frac{4\sqrt 3}{3} \leq r,
    \end{cases}\\
    \icech(X)_r &= \begin{cases}
        \{[a],[b],[c]\}, & 0 \leq r < 4, \\
        \{[a],[b],[c],[ab],[ac],[bc], [abc]\}, & 4 \leq r.
    \end{cases}
\end{align*}
\end{example}

\subsubsection{Weak Equivalence}\label{Sec:Weak_Equiv}
We frame our main results using \emph{weak equivalence}, a standard notion of homotopical equivalence of diagrams of spaces that has appeared in several TDA works \cite{alonsoDelaunayBifiltrations2023,blumberg2023universality,buchetSparseHigher2023,blumbergStability2Parameter2022,lanari2023rectification,corbetComputingMulticover2023}.  
For functors $\mathcal F, \mathcal F': \cat{P}\to \Top$, we say that a natural transformation $\eta:\mathcal F\to \mathcal F'$ is an \emph{objectwise homotopy equivalence} if for each $p\in P$, the component map $\eta_p:\mathcal F_p\to \mathcal F'_p$ is a homotopy equivalence.  If such an $\eta$ exists, we write \[\mathcal F\xlongrightarrow{\simeq}\mathcal F'.\]  Moreover, we say that $\mathcal F$ and $\mathcal F'$ are \emph{weakly equivalent}, and write $\mathcal F\simeq \mathcal F'$, if they are connected by a zigzag of objectwise homotopy equivalences, as follows:
\[
\begin{tikzcd}[ampersand replacement=\&,column sep=2ex,row sep=2ex]
   \& \mathcal W_1\ar["\simeq",swap]{dl}\ar["\simeq"]{dr}  \&           \& \cdots\ar["\simeq",swap]{dl}\ar["\simeq"]{dr}  \&                \&   \mathcal W_n \ar["\simeq",swap]{dl}\ar["\simeq"]{dr}  \\
\mathcal F \&                                                             \&  \mathcal W_2 \&                                                                                                                       \& \mathcal W_{n-1} \&                                                               \&\mathcal F'.
\end{tikzcd}
\]

\subsection{Persistence Modules}
\label{sec:Persistence}

Let $\Vec$ denote the category of vector spaces over a fixed field $\mathbb K$.  A \emph{($P$-)persistence module} is a functor $M:\cat P \to \Vec$.   If $v \in M_p$, then we call $p$ the \emph{grade} of $v$, and write $p=\gr v$. 
 
A \emph{generating set} of $M$ is a set $S \subset \bigsqcup_{p\in P}M_p$ such that for any $v \in \bigsqcup_{p\in P }M_p$ we have \[ v = \sum_{i=1}^k c_i M_{\gr{v_i}\subto \gr{v}}(v_i)\] for some vectors $v_1, v_2, \dots, v_k \in S$ and scalars $c_1, c_2, \dots, c_k \in \mathbb K$.   We say $S$ is \emph{minimal} if for all $v\in S$, the set $S\setminus \{v\}$ does not generate $M$.  An elementary linear algebra argument shows that if $S$ and $S'$ are minimal generating sets of $M$, then $|S|=|S|'$.  We say  $M$ is \emph{finitely generated} if a finite generating set of $M$ exists.  Clearly, if $M$ is finitely generated, then a minimal generating set of $M$ exists.

\subsubsection{Presentations and Resolutions}\label{Sec:Free_Modules}
For $p\in P$, let $\mathbb K^{\langle p\rangle}\colon P\to \Vec$ be given by 
\begin{align*}
\mathbb K^{\langle p\rangle}_x &=
\begin{cases}
\mathbb K &\text{if } p \leq x, \\
0 &\text{otherwise,}
\end{cases}
& \mathbb K^{\langle p\rangle}_{x\subto y}=
\begin{cases}
\Id_{\mathbb K} &{\text{if } p\leq x},\\
0 &{\text{otherwise}}.
\end{cases}
\end{align*}
A persistence module $F\colon P\to \Vec$ is \emph{free} if there exists a multiset $\mathcal B$ of elements in $P$ such that $F\cong \bigoplus_{p\in \mathcal B}\, \mathbb K^{\langle p\rangle}$.

A \emph{presentation} of $M\colon P\to \Vec$ is a morphism of free modules $\partial:F_1\to F_0$ such that $M\cong \mathrm{coker}(\partial)$.  $M$ is \emph{finitely presented} if there exists a presentation of $M$ with $F$ and $F'$ finitely generated.  
A \emph{free resolution} $F$ of $M$ is an exact sequence of free $P$-persistence modules \[\cdots \xrightarrow{\partial_3}  F_{2}\xrightarrow{\partial_2} F_1\xrightarrow{\partial_1} F_0\] such that $M\cong \mathrm{coker}(\partial_1)$.
$F$ is \emph{minimal} if any resolution of $M$ has a direct summand isomorphic to $F$.  

Now consider a finitely presented module $M\colon P\to \Vec$, where $P=T_1\times \cdots \times T_n$ for totally ordered sets $T_1,\ldots,T_n$.  Standard arguments show that there exists a minimal resolution $F$ of $M$ with each $F_i$ finitely generated, and that $F$ is unique up to isomorphism; see  \cite[Sections 1.7 and 1.9]{peeva2011graded}. We call the cardinality of a minimal set of generators for $F_i$ the \emph{$i^{\mathrm{th}}$ total Betti number of $M$}, and denote it as $\beta_i(M)$.

\subsection{Chain Complexes and Homology of Filtrations}\label{Sec:Chain_Complexes}
Let $H_i\colon \Top \to \Vec$ denote the $i^{\mathrm{th}}$ homology functor with coefficients in $\mathbb K$.  Given a poset $P$ and functor $\mathcal F\colon \cat{P} \to \Top$, we obtain a persistence module $\Ho{i}\mathcal F \colon \cat{P}\to \Vec$ by composition.  If $\mathcal F\simeq \mathcal F'$, then $\Ho{i}\mathcal F\cong \Ho{i} \mathcal F'$. In fact, $\Ho{i}{\mathcal F}$ is the homology of a chain complex of $P$-persistence modules
\[
C \mathcal F= \cdots \xrightarrow{\partial_{3}} \Ch{2}\mathcal F \xrightarrow{\partial_2} \Ch{1}\mathcal F\xrightarrow{\partial_1} \Ch{0}\mathcal F,
\]
where $(\Ch{i}\mathcal F)_p$ is the usual simplicial chain vector space $\Ch{i}{(\mathcal F_p)}$ with coefficients in $\mathbb K$, and each structure map $(\Ch{i}{\mathcal F})_{p\subto q}$ is the map induced by $\mathcal F_{p\subto q}$.  Note that \[\Ch{i}\mathcal F\cong \bigoplus_{\substack{\sigma\in \colim \mathcal F\\\dim(\sigma)=i}} \mathbb K^{\sigma},\] where \begin{align*}
\mathbb K^\sigma_p &=
\begin{cases}
\mathbb K &\text{if }  \sigma\in \mathcal F_p, \\
0 &\text{otherwise,}
\end{cases}
& \mathbb K^\sigma_{p\subto q}=
\begin{cases}
\Id_{\mathbb K} &{\text{if }  \sigma\in \mathcal F_p},\\
0 &{\text{otherwise}}.
\end{cases}
\end{align*}

A filtration $\mathcal F:P\to \Simp$ is called \emph{$1$-critical} if each $\Ch{i}{\mathcal F}$ is free.  
Note that $\mathcal F$ is $1$-critical if and only each simplex $\sigma\in \colim \mathcal F$ has a unique minimal index of appearance $\birth(\sigma)\in P$.    For example, Rips and Cech filtrations are both 1-critical.  

We will say a functor $\mathcal F\colon P\to \Simp$ is \emph{finitely presented} if each $\Ch{i}{\mathcal F}$ is finitely presented.  If $\mathcal F$ is a 1-critical filtration with finite colimit, then $\mathcal F$ is finitely presented.  Conversely, for $T$ a totally ordered set, a finitely presented filtration $\mathcal F\colon T\to \Simp$ is 1-critical and has finite colimit.

\subsubsection{Size of Simplicial Functors}\label{Sec:Size}

\begin{definition}\label{Def_Size_Filtration}
Let $P=T_1\times \cdots \times T_n$ for totally ordered sets $T_1,\ldots, T_n$.  We define the \emph{size} of a finitely presented functor $\mathcal F\colon P\to \Simp$ to be 
\begin{equation}\label{eq:Main_Size_Def}
\beta_1(\Ch{0}{\mathcal F})+\sum_{i=0}^\infty \beta_0(\Ch{i}{\mathcal F}).
\end{equation}
\end{definition}

\begin{remark}
We now give some context and motivation for \cref{Def_Size_Filtration}.  Outside of the familiar case of 1-critical filtrations, we are aware of only two TDA papers that quantify the size of a $\Simp$-valued functor: First, \cite[Section ~$1.2$]{lesnickInteractiveVisualization2015} defines the size of a finitely presented bifiltration $\mathcal F$ as \[\sum_{i=1}^\infty \beta_0(\Ch{i}{\mathcal F}),\]  which differs from  \cref{eq:Main_Size_Def} by at most a factor of $2$.  Second, for $T$ a totally ordered set, the main results of \cite{kerber2019barcodes} are formulated using several measures of the size of a functor $\mathcal F\colon T\to \Simp$; see \cref{Semi_Filtration_to_Bifilration}.

Our motivation for \cref{Def_Size_Filtration} is the following: Assuming that the simplicial complexes of $\mathcal F$ have bounded dimension and that storing an element of $\mathbb K$ requires constant memory, \eqref{eq:Main_Size_Def} is an asymptotic upper bound on the amount of memory needed to store $\mathcal F$.  To explain, the 0-skeleton of $\mathcal F$ can be encoded up to isomorphism using  \[O(\beta_0(\Ch{0}{\mathcal F})+\beta_1(\Ch{0}{\mathcal F}))\] memory, e.g., as a minimal presentation of $\Ch{0}{\mathcal F}$.  Since the structure maps of $\mathcal F$ are simplicial, they are determined by the 0-skeleton of $\mathcal F$.  Therefore, to store the rest of $\mathcal F$, we need only to record all births of higher-dimensional simplices in the filtration, which requires \[O\left(\sum_{i=1}^\infty \beta_0(\Ch{i}{\mathcal F})\right)\] memory.

For practical computations of homology, it is natural to begin with explicit presentations of the chain modules. Thus, one might prefer to instead define the size of $\mathcal F$ to be 
\begin{equation}\label{eq:Other_Size_Def}
\sum_{i=1}^\infty  \beta_1(\Ch{i}{\mathcal F})+\sum_{i=1}^\infty \beta_0(\Ch{i}{\mathcal F}).
\end{equation}
However, we are primarily interested in the case where $\mathcal F$ is a semifiltration; in this case, \cref{Prop:Betti_Numbers_Semi_Filtrations} below gives that  $\beta_1(\Ch{i}{\mathcal F})\leq \beta_0(\Ch{i}{\mathcal F})$ for all $i\geq 0$, implying that the measures of size \eqref{eq:Main_Size_Def} and \eqref{eq:Other_Size_Def} differ by at most a factor of $2$. 
\end{remark}

\subsection{Subdivision Bifiltrations}
\label{sec:SRips}

A sequence of nested simplices \[\sigma_1 \subset \sigma_2 \subset \cdots \subset \sigma_m\] of an abstract simplicial complex $W$ is called a \emph{flag} of $W$. The set $\Bary{W}$ of all flags of $W$ is itself a simplicial complex, called the \emph{barycentric subdivision} of $W$.  
 
 \begin{definition}
For each $k \in \N\coloneqq \{1,2,\ldots\}$, define $\mathcal S(W)_k$ to be the subcomplex of $\Bary W$ spanned by all flags whose minimum element has dimension at least $k-1$.  Then in particular, $\mathcal S(W)_1 = \Bary W$.  Varying $k$ yields a filtration
\[
\mathcal S(W)\colon \mathbb N^{\mathrm{op}}\to \Simp;
\]
 see \Cref{fig:subdiv-2spx}.  
\begin{figure}[t]
\centering
\includegraphics[width=0.75\textwidth]{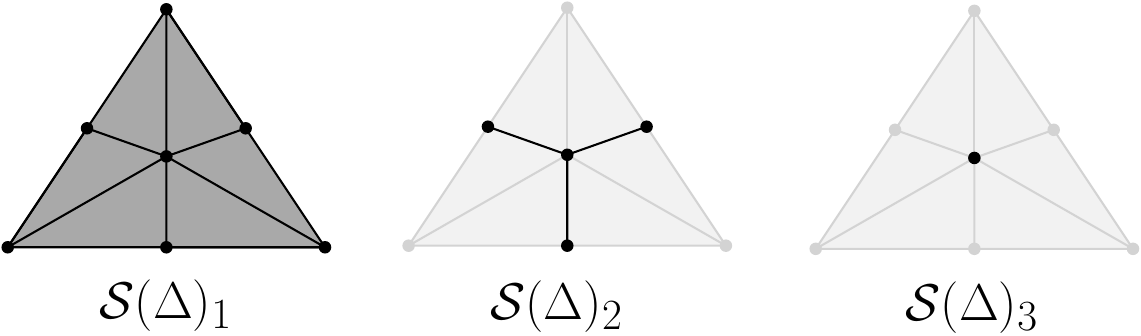}
\caption{The subdivision filtration of a $2$-simplex $\Delta$.}
\label{fig:subdiv-2spx}
\end{figure}
For any poset $P$ and filtration ${\mathcal F:\cat{P}\to \Simp}$, the family of filtrations $(\mathcal S(\mathcal F_p))_{p\in P}$ assembles into a filtration
\[
\subdiv \colon \mathbb N^{\mathrm{op}}\times P \to \Simp,
\]
which we call the \emph{subdivision filtration of $\mathcal F$} \cite{sheehyMulticoverNerve2012a}.  In this paper, we are primarily interested in the case $P=[0,\infty)$.  For $X$ a metric space, we call $\srips(X)$ and $\mathcal S\icech(X)$ the \emph{subdivision-Rips} and \emph{intrinsic subdivision-\v Cech bifiltrations of $X$}, respectively. 
\end{definition}

Note that if $\mathcal F$ is 1-critical, then so is $\subdiv$, and that if $\mathcal F$ is finitely presented, then so is $\subdiv$.

\subsection{Homotopy Interleavings}
\label{sec:Interleavings}

A category is said to be \emph{thin} if for any objects $x$ and $y$, there is at most one morphism from $x$ to $y$.  For $\eps \geq 0$, let  $I^{1+\eps}$ be the thin category with object set $\T\times \{0,1\}$ and a morphism $(r,i)\to (s,j)$ if and only if either
\begin{itemize}
    \item [(i)] $r(1+\eps) \leq s$, or
    \item [(ii)] $i=j$ and $r \leq s$.
\end{itemize}

We then have functors $E^0, E^1:\T\to I^{1+\eps}$ mapping $r \in\T$ to $(r,0)$ and $(r,1)$, respectively.

\begin{definition}
    For any category $\mathbf C$ and functors $\mathcal F, \mathcal F':\T\to \mathbf C$, a (multiplicative) $(1+\eps)$-interleaving between $\mathcal F$ and $\mathcal F'$ is a functor 
    \[
    \mathcal{Z}: I^{1+\eps} \to \mathbf C
    \]
    such that 
    $  \mathcal{Z}\circ E^0 = \mathcal F$ and $\mathcal{Z}\circ E^1= \mathcal F'$.
If such a $\mathcal Z$ exists, we say $\mathcal F$ and $\mathcal F'$ are \emph{$(1+\epsilon)$-interleaved}.  
\end{definition}

We now extend this definition to the $2$-parameter setting: For $\eps \geq 0$, let the $I^{(1,1+\eps)}$, be the thin category with object set $\Poset\times \{0, 1\}$ and a morphism $(k,r,i)\to (l,s,j)$ if and only if either
\begin{itemize}
    \item [(i)] $(k,r(1+\eps)) \leq (l,s)$, or 
    \item [(ii)] $i=j$ and $(k,r) \leq (l,s)$.
\end{itemize}
As above, we have functors $E^0, E^1: \Poset \to I^{(1,1+\eps)}$ sending $(k,r) \in \Poset$ to $(k,r,0)$ and $(k,r,1)$, respectively.

\begin{definition}\label{def:2d-interleaving}
    For functors $\mathcal F, \mathcal F': \Poset \to \mathbf C$, a $(1+\eps)$-interleaving between $\mathcal F$ and $\mathcal F'$ is a functor 
    \[
      \mathcal{Z}: I^{(1,1+\eps)} \to \mathbf C
    \]
    such that $\mathcal{Z}\circ E^0 = \mathcal F$ and $  \mathcal{Z}\circ E^1 = \mathcal F'$.
\end{definition}

\begin{remark}
\cref{def:2d-interleaving} is essentially the same as the definition of 2-parameter interleaving used in \cite{buchetSparseHigher2023} to formulate sparse approximation results for the multicover bifiltration.  It is somewhat different than the definition of interleavings used in \cite{blumbergStability2Parameter2022} to formulate \cref{Thm:Robustness}; the definition of \cite{blumbergStability2Parameter2022} also allows for shifts in the first coordinate, and the shifts in each coordinate are additive rather than multiplicative.  
\end{remark}

\begin{definition}[\cite{blumberg2023universality}]
    For  functors $\mathcal F, \mathcal G: \Poset \to \Top$, we say that $\mathcal F$ and $\mathcal G$ are \emph{$(1+\eps)$-homotopy interleaved} if there exist $(1+\eps)$-interleaved  functors $\mathcal F^\prime, \mathcal G^\prime$ such that $\mathcal F \simeq \mathcal F^\prime$ and $\mathcal G\simeq \mathcal G^\prime$.
\end{definition}

If $\mathcal F$ and $\mathcal G$ are $(1+\eps)$-homotopy interleaved, we say that  $\mathcal G$ is a \emph{$(1+\eps)$-approximation} to $\mathcal F$.

\begin{proposition}\label{Prop:Inerleaving_and_Subdivision}
For any $\epsilon\geq 0$, if filtrations $\mathcal F,\mathcal F'\colon \T\to \Simp$ are $(1+\eps)$-interleaved, then so are $\mathcal{SF}$ and $\mathcal{SF'}$.
\end{proposition}

\begin{proof}
For $\eps=0$, the result holds because isomorphic filtrations have isomorphic subdivision filtrations.  For the case $\eps>0$, consider a $(1+\eps)$-interleaving $\mathcal Z$ between $\mathcal F$ and $\mathcal F'$.  Note that since the structure maps of $\mathcal F$ and $\mathcal F'$ are inclusions, the structure maps of $\mathcal Z$ are monomorphisms  (i.e., injections on vertex sets).  Hence, $\mathcal Z$ is naturally isomorphic to a filtration $\mathcal Z'$, i.e., a functor whose structure maps are inclusions; concretely, we can take each $\mathcal Z'_r$ to be the image of $\mathcal Z_p$ under the cocone map $\mathcal Z_r\hookrightarrow \colim \mathcal Z$.  The functor $\mathcal{SZ'}$ is a $(1+\eps)$-interleaving between filtrations isomorphic to $\mathcal{SF}$ and $\mathcal {SF'}$.  It follows that $\mathcal{SF}$ and $\mathcal {SF'}$ are themselves $(1+\eps)$-interleaved.
\end{proof}

\subsection{Functorial Nerves}\label{subsec:nerve-model}
Following \cite{bauerUnifiedView2023}, we give a functorial version of the nerve theorem for simplicial covers.  We begin with the requisite definitions.  

For $X$ a topological space, a \emph{cover of $X$} is a collection $U=(U_a)_{a\in A}$ of subspaces of $X$ indexed by a set $A$ such that $\bigcup_{a\in A} U_a=X$.  We say $U$ is \emph{good} if the intersection of finitely many elements of $U$ is always either empty or contractible.

Let $X$ and $Y$ be topological spaces, $U=(U_a)_{a\in A}$ a cover of $X$, and $V=(V_b)_{b \in B}$ a cover of $Y$.  We define a \emph{map of cover indices} $\varphi\colon U\to V$ to be a function $\varphi\colon A\to B$.  We say that $\varphi$ \emph{carries} a continuous map $f\colon X\to Y$ if for all $a \in A$ we have $f(U_a)\subset V_{\varphi(b)}$.

\begin{definition}
The objects of the \emph{category of covered spaces} $\cCov$ are pairs $(X,U)$ where $X$ is a topological space and $U$ is a cover of $X$. A morphism $(f, \varphi):(X,U)\to (Y,V)$ in $\cCov$ consists of a  continuous map $f\colon X\to Y$ and a map of cover indices $\varphi\colon U\to V$ such that $\varphi$ carries $f$.  
\end{definition}

The nerve construction (\cref{Def:Nerve}) induces a functor $\Ner\colon \cCov \to \Simp$.  We also have a forgetful functor $\mathcal Y\colon\cCov \to \Top$ that maps an object $(X,U)$ to $X$.  

Using these functors, we can extend the formalism of covers and nerves to diagrams of spaces, as follows:
\begin{definition}\label{def:functor-nerve}
    Let $\mathcal F\colon P \to \Top$ be a functor. 
    \begin{itemize}
    \item[(i)]
    A \emph{cover} of $\mathcal F$ is a functor $\mathcal U\colon P\to \cCov$ such that $\mathcal Y \circ \mathcal U = \mathcal F$. 
    \item[(ii)] Writing $\mathcal U_p=(X_p,U_p)$, we say that $\mathcal U$ is  \emph{good} if $U_p$ is good for each $p\in P$. 
    \item[(iii)] The \emph{nerve of $\mathcal U$} is the composite $\Ner{}\! \circ \mathcal U$, which we will denote as $\Ner_{\mathcal U}$. 
    \end{itemize}
\end{definition}

The following is the version of functorial nerve theorem we will use in this paper.  See \cite{bauerUnifiedView2023} for a lucid treatment of variants and generalizations.

\begin{theorem}[{\cite[Theorem 4.8]{bauerUnifiedView2023}}]\label{theo:functorial-nerve-theorem}
If $\mathcal U$ is a good simplicial cover of a functor $\mathcal F\colon P \to \Simp$, then $\Ner_{\mathcal U} \simeq \mathcal F$.
\end{theorem}

\section{Size of Models of the Subdivision Bifiltration}
\label{sec:nerve-model}

In this section, we define our nerve-based model $\mathcal{NF}$ of a subdivision bifiltration $\mathcal{SF}$, and prove \cref{thm:main-theorem-intro}.  First, we recall the statement of the theorem, along with the requisite notation. 

\begin{repnotation}{Notation:m_k}
For $T$ a totally ordered set, $\mathcal F\colon T\to \Simp$ a filtration and $k\geq 0$, let $m_k=m_k(\mathcal F)$ denote the number of sets $S$ of simplices in  $\colim \mathcal F$ such that 
\begin{itemize}
\item $|S|\leq (k+1)$ and
\item  for some $t\in T$, each $\sigma\in S$ is a maximal simplex in $\mathcal F_t$.
\end{itemize}
\end{repnotation}

\begin{reptheorem}{thm:main-theorem-intro}\mbox{}
For any finitely presented filtration $\mathcal F\colon T\to \Simp$,
    \begin{itemize}
        \item [(i)] $\mathcal{NF}$ is a simplicial model of $\subdiv$ whose $k$-skeleton has size $O(m_k)$,
        \item [(ii)] the $0$-skeleton of any simplicial model of $\subdiv$ has size at least $m_0$.
    \end{itemize}
\end{reptheorem}

\subsection{Nerve Model Construction}\label{Sec:Nerve_Model_Subdiv}
Given a poset $P$ and filtration $\mathcal F\colon P\to \Simp$, we now define $\mathcal{NF}\colon \mathbb N^{\op}\times P\to \Simp$.  As for subdivision filtrations, we are primarily interested in the case $P=[0,\infty)$, but it is no more difficult to give the construction for general $P$.  

We first choose a well-ordering of the vertices of $\colim \mathcal F$.  Observe that the resulting lexicographic ordering of simplices in $\colim \mathcal F$ is also a well-ordering.  Let $\Sigma$ denote the set of simplices in $\colim \mathcal F$  that are maximal in $\mathcal F_t$ for some $t\in P$, and let $\Sigma_t\subset \Sigma$ denote the set of maximal simplices in $\mathcal F_t$.  For $t\leq t'\in P$, define  $\Sigma_{t\subto t'}\colon \Sigma_{t}\to \Sigma_{t+1}$ to be the map sending $\sigma\in \Sigma_t$ to the lexicographically minimum simplex in $\Sigma_{t'}$ containing $\sigma$.  
The sets $(\Sigma_t)_{t\in P}$ and maps $(\Sigma_{t\subto t'})_{t\leq t'\in P}$ form a cover of $\mathcal F$, which induces a cover $\mathcal U$ of $\subdiv$ via barycentric subdivision.  Concretely, \[\mathcal U_{k,t}=\{\mathcal S(\bar \sigma)_k\mid \sigma \in \Sigma_t\},\] where $\bar \sigma$ is the simplicial complex consisting of $\sigma$ and all of its faces.  We define $\mathcal{NF}$ to be $\Ner_{\mathcal U}$.    
See \Cref{Fig:Main} for an illustration.

\begin{figure}
    \begin{subfigure}[b]{.3\linewidth}
    \centering
        \includegraphics[height=5.5cm]{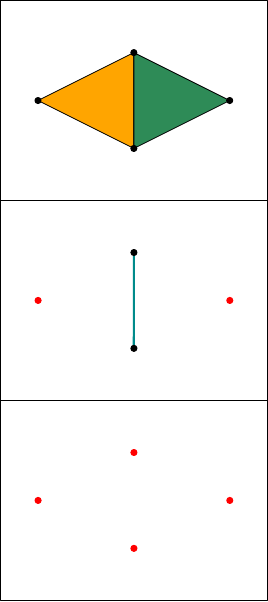}
        \caption{Part of the filtration $\rips(X)$, with maximal simplices colored.}
        \vspace{1cm}
        \label{Subfig:Undivided}
    \end{subfigure}
    \hskip20pt
    \begin{subfigure}[b]{.6\linewidth}
    \centering
        \includegraphics[height=5.5cm]{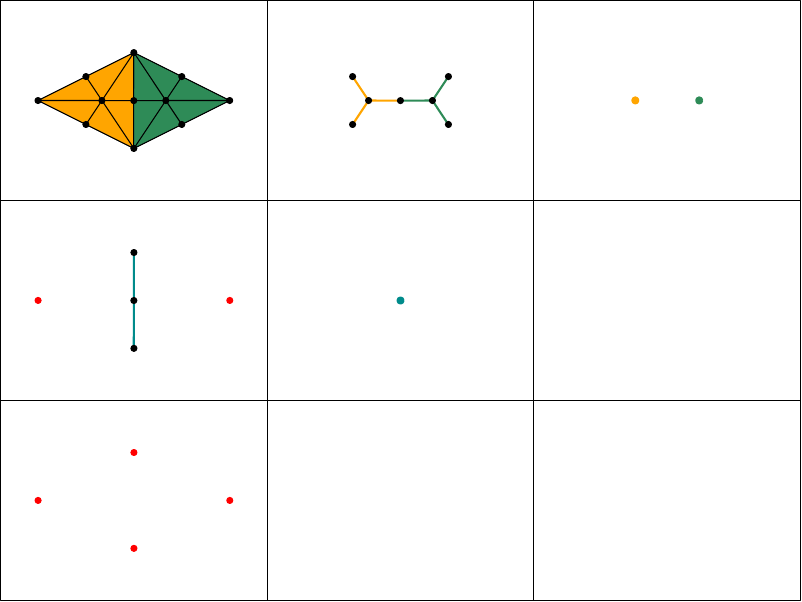}
        \caption{The bifiltration $\srips(X)$, with each maximal simplex given the color of its corresponding undivided maximal simplex in (A).}
        \vspace{1cm}
        \label{fig:main2}
    \end{subfigure}
    \begin{subfigure}[b]{\linewidth}
    \centering
        \includegraphics[height=5.5cm]{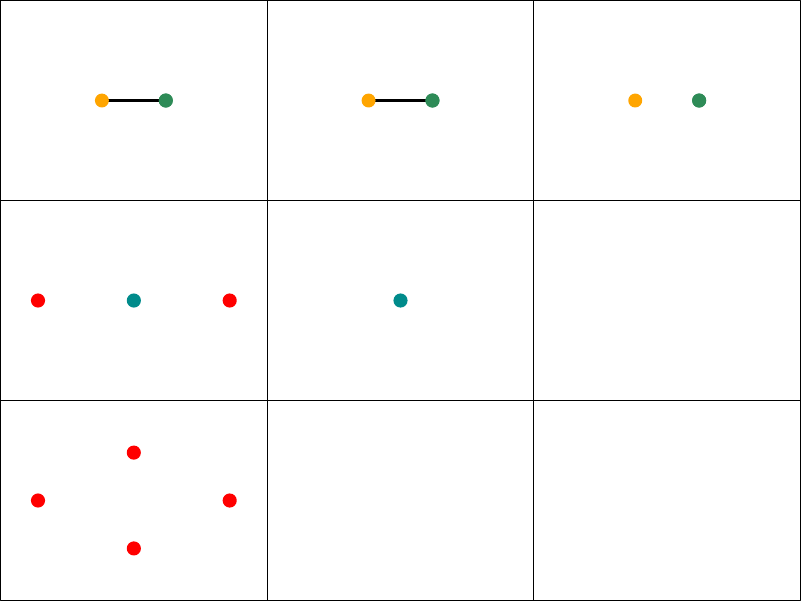}
        \caption{The semifiltration $\mathcal{NR}(X)$, with each vertex given the color of its corresponding maximal simplex in (A).  Whether the blue vertex in the center panel maps to orange or to the green vertex in the center-top panel depends on a choice of order on $X$: The blue vertex maps to the orange one if and only if the leftmost point of $X$ is ordered before the rightmost point.}
    \end{subfigure}
    \caption{Construction of $\mathcal{NR}(X)$, for $X$ a set of four points in $\R^2$.}
        \label{Fig:Main}
\end{figure}

\begin{proposition}\label{cor:nerve-srips-equiv}\sloppy
   For any poset $P$ and filtration ${\mathcal F\colon P\to \Simp}$, we have $\mathcal{NF} \simeq \subdiv$.
\end{proposition}

To prove \Cref{cor:nerve-srips-equiv}, we will need the following lemmas.

\begin{lemma}\label{lem:spx-contract}
Let $\Delta$ be a simplex, regarded as a simplicial complex.  For any $k \in \mathbb N^{\op}$, $\mathcal S(\Delta)_k$ is etither empty or contractible.
\end{lemma}

\begin{proof}
Observe that $\mathcal S(\Delta)_k$ is the order complex of a subposet $Q$ of the face poset of $\Delta$.  If $Q$ is not empty, then it has a maximal element, namely,  the top-dimensional simplex of $\Delta$.  As the order complex of a poset with a maximal element is contractible \cite{wachsPoset07}, the result follows.
\end{proof}

See \cref{fig:subdiv-2spx} for an illustration of \cref{lem:spx-contract} in the case $j=2$.   We omit the easy proof of the next lemma.

\begin{lemma}\label{lem:bary-int}
    Let $W$ be a finite simplicial complex, and let $\{\sigma_i\}$ be a collection of simplices of $W$. Then for any $k \in \mathbb N^{\op}$,
    \[
    \bigcap_{i} \mathcal S(\sigma_i)_k = \mathcal S \left( \bigcap_{i} \sigma_i \right)_k.
    \]
\end{lemma}

\begin{proof}[Proof of \cref{cor:nerve-srips-equiv}]
By  \Cref{theo:functorial-nerve-theorem}, it suffices to show that  $\mathcal U$ is a good cover of $\subdiv$, i.e., that for any $(k,t) \in \N^{\op}\times P$ and finite $U \subset \mathcal U_{k,t}$, the intersection of elements in $U$ is either empty or contractible.  Write $U=\{U^1,U^2, \dots,U^l\}$, and let
    \[
    \overline U := \bigcap_{i=1}^l U^i.
    \]
    If $\overline U= \emptyset$, we are done, so suppose $\overline U$ is nonempty. For each $i$, we have $U^i = \mathcal S(\sigma_i)_k$ for some $\sigma_i \in \Sigma_t$. Since any nonempty intersection of simplices in a simplicial complex is a simplex, we have 
    \[
    \bigcap_{i=1}^l \sigma_i = \Delta
    \]
    for some simplex $\Delta$. Applying \Cref{lem:bary-int}, we then get
    \[
    \overline U = \bigcap_{i=1}^l \mathcal S(\sigma_i)_k = \mathcal S \left( \bigcap_{i=1}^l \sigma_i \right)_k = \mathcal S(\Delta)_k,
    \]
    which is contractible by \Cref{lem:spx-contract}.
\end{proof}

\begin{lemma}\label{lem:semifiltration}
    For any totally ordered set $T$ and filtration $\mathcal F\colon T\to \Simp$, $\mathcal{NF}$ is a semifiltration.
\end{lemma}

\begin{proof}
    Fix $t\in T$ and let $j \leq k$ in $\N^{\mathrm{op}}$. Writing $\mathcal N=\mathcal{NF}$, we have 
    \begin{align*}
        \mathcal {N}_{j,t} &= \Ner \{\mathcal S(\sigma)_j \mid \sigma \in \Sigma_t\},\\
        \mathcal {N}_{k,t} &= \Ner \{\mathcal S(\sigma)_k \mid \sigma \in \Sigma_t\}.
       \end{align*}
    Since $\mathcal S(\sigma)_j \subset \mathcal S(\sigma)_k$ for every $\sigma \in \Sigma_t$, we have 
  $\mathcal{N}_{j,t} \subset \mathcal N_{k,t}.$
\end{proof}

\begin{proposition}\label{Prop:Betti_Numbers_Semi_Filtrations}
Let $\mathcal A\colon \N^2 \to \Simp$ be a finitely presented semifiltration.  For any $j\geq 0$, we have $\beta_1(\Ch{j} \mathcal A)\leq \beta_0(\Ch{j} \mathcal A)$, and $\beta_i(\Ch{j}\mathcal A)=0$ for $i>1$. 
\end{proposition}

\begin{proof}
Since $\Ch{j}\mathcal A$ is finitely generated, Hilbert's syzygy theorem implies that $\beta_i(\Ch{j}\mathcal A)=0$ for $i>2$ \cite[Theorem 15.2]{peeva2011graded}.  To see that $\beta_2(\Ch{j}\mathcal A)=0$, we use the well-known characterization of Betti numbers in terms of Koszul homology; see, e.g., \cite[Proposition 2.7]{eisenbud2006geometry} for the $\Z$-graded case.    For a finitely generated persistence module $M\colon 
\mathbb N^2 \to \Vec$, the Koszul homology formula for $2^{\mathrm{nd}}$ Betti numbers implies that
\[\beta_2(M)=\sum_y \sum_z \Nullity(M_{y-1,z-1}\xrightarrow{f_{y,z}} M_{y,z-1}\oplus M_{y-1,z}),\] where the map $f_{y,z}$ is induced by the structure maps of $M$.  Since all of the horizontal structure maps of $\Ch{j}\mathcal A$ are injections, this formula yields $\beta_2(\Ch{j}\mathcal A)=0$.  

According to a standard inclusion-exclusion formula relating Betti numbers and the Hilbert function \cite[Theorem 16.2]{peeva2011graded}, for $y,z$ sufficiently large we have \[\dim(M_{y,z})=\beta_0(M)-\beta_1(M)+\beta_2(M).\]  Since $\dim M_{y,z}\geq 0$, this implies that $\beta_0(M)\geq \beta_1(M)-\beta_2(M)$.  Since $\beta_2(\Ch{j}\mathcal A)=0$, we have $\beta_1(\Ch{j}\mathcal A)\leq \beta_0(\Ch{j}\mathcal A)$.
\end{proof}

To bound the size of $\mathcal{NF}$ in \cref{Cor:Thm_Generalization}\,(i), we will use the $j=0$ case of the following variant of \cref{Prop:Betti_Numbers_Semi_Filtrations}:
\begin{corollary}\label{Prop:Betti_Numbers_Semi_Filtrations_Other_Indices}
Let $T$ be a totally ordered set and let $\mathcal A\colon \mathbb N^{\op}\times T\to \Simp$ be a finitely presented semifiltration.  For any $j\geq 0$, we have $\beta_1(\Ch{j} \mathcal A)\leq \beta_0(\Ch{j} \mathcal A)$, and $\beta_i(\Ch{j}\mathcal A)=0$ for $i>1$. 
\end{corollary}

\begin{proof}
Our argument uses ideas similar to those in  \cite[Section 2.6]{bjerkevik2021ell}, wherein the reader may find additional details.  Define a \emph{grid function} to be a poset map $g\colon \N^2 \to \N^{\op}\times T$ of the form $g(y,z)=(g_1(y),g_2(z))$, for maps $g_1\colon \N\to  \N^{\op}$ and  $g_2\colon \N\to T$.  Since $\mathcal A$ is finitely presented, $\mathcal A$ is the left Kan extension of a functor $\mathcal A'\colon \mathbb N^2\to \Simp$ along a grid function.  

Left Kan extension along $g$ is functorial, i.e., defines a functor \[\mathcal{E}_g(-)\colon \mathbf C^{\mathbb N^2}\to \mathbf C^{\mathbb N^{\op}\times T}\] for any category $\mathbf C$.  The functor $\mathcal{E}_g(-)$ commutes with the construction of simplicial chain complexes, so that in particular, we have 
\[C\mathcal A=C(\mathcal{E}_g(\mathcal A'))=\mathcal{E}_g(C\mathcal A').\]  Moreover, $\mathcal{E}_g(-)$ is exact when $\mathbf C=\Vec$ \cite[Lemma 2.23]{bjerkevik2021ell}.  Thus, for any persistence module 
$M\colon \N^2\to \Vec$, the functor $\mathcal{E}_g(-)$ maps a resolution of $M$ to a resolution of $\mathcal{E}_g(M)$ of the same size.    
In particular, taking $M=\Ch{j}\mathcal A'$, this implies that $\mathcal{E}_g(-)$ maps a resolution of  $\Ch{j}\mathcal{A'}$ to a resolution of $\Ch{j}\mathcal{A}$ of the same size.  Conversely, the restriction functor along $g$ is also exact, and so maps a resolution of $\Ch{j}\mathcal{A}$ to a resolution of $\Ch{j}\mathcal{A'}$ of the same size.   Therefore, for each $i,j\geq 0$ we have $\beta_i(\Ch{j}\mathcal{A})=\beta_i(\Ch{j}\mathcal{A}')$.  This, together with \Cref{Prop:Betti_Numbers_Semi_Filtrations}, yields the result.  
\end{proof}

\begin{proof}[Proof of \Cref{Cor:Thm_Generalization}\,(i)]
\cref{cor:nerve-srips-equiv} tells us that $\mathcal{NF}$ is a simplicial model of $\subdiv$.  Recalling our definition of size of a finitely presented $\Simp$-valued functor (\cref{Def_Size_Filtration}) and appealing to the $j=0$ case of  \cref{Prop:Betti_Numbers_Semi_Filtrations_Other_Indices}, it suffices to show that for all non-negative integers $k$, we have
\begin{equation}\label{eq:m_k_bound}
\sum_{j=0}^k  \beta_0(C_j(\mathcal{NF}))\leq m_k.
\end{equation}
Recall the definitions of $\Sigma$ and $\Sigma_t$ from the beginning of this subsection.  For $j$ a non-negative integer, let 
\[\Sigma^j=\{S\subset \Sigma\mid |S|=j+1,\ S\subset \Sigma_t \textup{ for some }t\in T,\ \bigcap_{\sigma\in S} \sigma \ne \emptyset\}.\]
For $S\in \Sigma^j$, let $\birth(S)=\max_{\sigma\in S} \birth(\sigma)$, and $\dim S=\dim (\bigcap_{\sigma\in S} \sigma)$.  Let 
$G\colon \mathbb N^{\op}\times T\to \Vec$ be the free persistence module with basis the multiset \[B^j\coloneqq \{(\dim S+1,\birth(S))\mid S \in \Sigma^j)\}.\] We have an evident surjection $G\twoheadrightarrow C_j(\mathcal{NF})$.   Hence, 
\[\beta_0(C_j(\mathcal{NF}))\leq |B^j|=|\Sigma^j|.\]   Finally, it follows from the definitions that \[\sum_{j=0}^k |\Sigma^j|\leq m_k,\]
which establishes the inequality \eqref{eq:m_k_bound}.   
\end{proof}

\subsection{Zeroth Homology}
\label{sec:ZerothHomology}
The key to our proof of \cref{Cor:Thm_Generalization}\,(ii) is the following proposition.

\begin{proposition}\label{Prop_H0_Bound}
For any totally ordered set $T$ and  finitely presented filtration $\mathcal F:T \to \Simp$, we have $\beta_0( H_0(\subdiv))=m_0$.
\end{proposition}

We prepare for the proof of the proposition with a lemma.  For a simplex $\sigma\in \mathcal F_t$ we let $\hat \sigma$ denote the vertex corresponding to $\sigma$ in $\Bary{\mathcal {F}_t}$.

\begin{lemma}\label{guacaroni}
If $\sigma \in \Sigma$, then $\hat{\sigma}$ is an isolated component of $\subdiv_{\birth(\hat \sigma)}$.
\end{lemma}

\begin{proof}
Let $\birth (\hat \sigma) = (k,z)$.  Note that $k=\dim \sigma + 1$ and that $z$ is the minimum index such that  $\sigma \in\mathcal F_z$.  Seeking a contradiction, suppose that $\hat \sigma$ belongs to some $1$-simplex $[\hat \sigma,\hat \tau] \in \subdiv_{k,z}$.    We have $\subdiv_{k,z}\subset \mathcal F^+_z$, so by the definition of barycentric subdivision, simplices in $\subdiv_{k,z}$ correspond to flags in $\mathcal F_z$.  Thus,  either $\tau \subset \sigma$ or vice-versa.  If $\tau \subset \sigma$, then $\dim \tau < \dim \sigma=k-1$, implying that  $\hat \tau\not\in  \subdiv_{k,z}$, a contradiction.    If $\sigma \subset \tau$, then $\sigma$ is not maximal in $\mathcal F_z$, contradicting the assumption that $\sigma\in \Sigma$. 
\end{proof}

\begin{proof}[Proof of \cref{Prop_H0_Bound}]
Let $\hat \Sigma=\{\hat \sigma\mid \sigma\in \Sigma\}$.  Identifying each  $\hat\sigma\in \hat \Sigma$ with its corresponding element in $H_0(\subdiv)_{\birth(\hat\sigma)}$, we will show that $\hat \Sigma$ is a minimal set of generators for $H_0(\subdiv)$.  To see that $\hat \Sigma$ generates $H_0(\subdiv)$, it suffices to show that for any vertex $\hat \tau\in \colim \mathcal F^+$, either $\hat\tau\in \hat \Sigma$ or there exists a 1-simplex $[\hat \sigma,\hat \tau]\in  \subdiv_{\birth(\hat \tau)}$ with $\hat \sigma\in \hat \Sigma$.  Suppose that $\hat\tau\not\in \hat \Sigma$, and let $\sigma$ be a maximal coface of $\tau$ in $\mathcal F_{\birth(\tau)}$.  Note that $\sigma$ properly contains $\tau$ and that $\sigma\in \Sigma$.  We have $\birth(\tau)=\birth(\sigma)$ and $\dim \sigma>\tau$, which implies that $[\hat \sigma,\hat \tau]\in  \subdiv_{\birth(\hat \tau)}$.  This shows that $\hat \Sigma$ generates $H_0(\subdiv)$.  \cref{guacaroni} now implies that $\hat \Sigma$ in fact minimally generates $H_0(\subdiv)$.  
\end{proof}

\begin{proof}[Proof of \cref{Cor:Thm_Generalization}\,(ii)]
 Let $\mathcal F'$ be a simplicial model of $\subdiv$, and let $\mathcal{F}'^{\,0}$ denote the $0$-skeleton of $\mathcal F'$. By definition, we have \[|\mathcal{F}'|\geq |\mathcal{F}'^{\, 0}|\geq \beta_0( \Ch{0}{\mathcal F'}),\] where $|\cdot|$ denotes size. There is a canonical minimal generating set for $\Ch{0}{\mathcal F'}$, namely, the set
 \[
 Z:=\bigsqcup_{z \in \N^{\mathrm{op}}\times T} \{ \sigma \in \mathcal{F}'^{\, 0}_z \mid \sigma \notin \Img \mathcal{F'}_{y,z} \text{ for all } y < z \}.
 \]
Since $\Ho{0}\mathcal F'$ is a quotient of $\Ch{0}{\mathcal F'}$, a subset of $Z$ descends to a minimal generating set for $\Ho{0}\mathcal F'$, so $\beta_0( \Ch{0} {\mathcal F'})\geq \beta_0( \Ho{0}\mathcal F')$.  Since homology is invariant under weak equivalence, we have that $\Ho{0}\mathcal F' \cong H_0(\subdiv)$, so $\beta_0( \Ho{0}\mathcal F' )= \beta_0( H_0(\subdiv))$.  Thus, we have 
    \[
   |\mathcal{F}'|\geq  \beta_0( \Ch{0} {\mathcal F'}) \geq \beta_0( \Ho{0}\mathcal F' )= \beta_0( H_0(\subdiv)) = m_0,
    \]
    where the last equality follows from \Cref{Prop_H0_Bound}.
\end{proof}

\section{Approximation via Intrinsic Subdivision-\v{C}ech}\label{sec:intrinsic-cech}
We next prove \cref{cor:intro-3}, which concerns poly-size approximation of the subdivision-Rips bifiltration of an arbitrary finite metric space.  Recall the statement:
\begin{repcorollary}{cor:intro-3}\mbox{}
    \begin{itemize}
 \item[(i)] For any finite metric space $X$, there exists a simplicial $\sqrt{2}$-approximation $\mathcal J(X)$ to $\srips(X)$ whose $k$-skeleton has size $O(|X|^{k+2})$.   
 \item[(ii)]  There exists an infinite family of finite metric spaces $X$ such that for any $\delta\in [1,\sqrt{2})$, there is no $\delta$-approximation to $\srips(X)$ with size polynomial in $|X|$.
     \end{itemize}
\end{repcorollary}

Define $\icech^\prime (X)\colon \T\to \Simp$ by 
    \[
    \icech^\prime(X)_{r} \coloneqq \icech(X)_{\sqrt{2}r},
    \]
where $\mathcal I(X)$ is the intrinsic \v Cech filtration of \cref{Sec:Rips_Cech}.  In \cref{cor:intro-3}\,(i), we take $\mathcal J(X)\coloneqq \mathcal{NI}'(X)$.  

For the next lemma, recall the definition of $m_k$ from \cref{Notation:m_k}.
\begin{lemma}\label{lem:cech-simplices}
For any finite metric space $X$ and $k\geq 0$, $m_k(\icech(X))\leq |X|^{k+2}$.
\end{lemma}

\begin{proof}
A simplex $\sigma\in \icech(X)_r$ is maximal in $\icech(X)_r$ only if there is some point $x\in X$ such that $\sigma=B(x)_r$, i.e., $\sigma$ is the closed ball of radius $r$ in $X$ centered at $x$. 
We call $x$ a \emph{witness} of $\sigma$.  For $S=\{\sigma_1,\ldots,\sigma_k\}$ a set of maximal simplices in $\icech(X)_r$, we say that 
 $\{x_1,\ldots,x_k\}$ \emph{witnesses} $S$ if each $x_i$ witnesses $\sigma_i$.  Call any set with $j$ elements a \emph{$j$-set}. 
 
 For each $x\in X$, there are at most $|X|$ balls in $X$ centered at $x$.  Thus, as $r$ varies, each $(k+1)$-set in $X$ witnesses a total of at most $(k+1)|X|$ distinct $(k+1)$-sets of maximal simplices.  The number of $(k+1)$-sets contained in $X$ is ${{|X|}\choose{k+1}}$.  Since \[{{|X|}\choose{k+1}}\leq \frac{|X|^{k+1}}{k+1},\] it follows that $m_k(\icech(X))\leq |X|^{k+2}$.
\end{proof}

\begin{proof}[Proof of \cref{cor:intro-3}\,(i)]
Since the metric on $X$ satisfies the triangle inequality, for each $r \in\T$ we have
    \[
    \icech(X)_r \subset \rips (X)_r \subset \icech(X)_{2r}, 
\]
or equivalently, 
    \[
    \icech'(X)_{r/\sqrt{2}} \subset \rips (X)_r \subset \icech'(X)_{\sqrt{2} r}. 
\]
Hence $\icech^\prime (X)$ and $\rips (X)$ are $\sqrt{2}$-interleaved.  \cref{Prop:Inerleaving_and_Subdivision} then implies that  $\mathcal S\icech^\prime (X)$ and $\mathcal{SR} (X)$ are $\sqrt{2}$-interleaved as well.  

Since $\icech'(X)$ is a rescaling of $\icech(X)$, \cref{lem:cech-simplices} implies that $m_k(\icech'(X))=O(|X|^{k+2})$.   Applying \cref{Cor:Thm_Generalization}\,(i) with $\mathcal F=\icech'(X)$ yields the result. \end{proof}

We now turn to the proof of \cref{cor:intro-3}\,(ii).  

\begin{definition}\label{Def:Cocktail_Party_Graph} 
For $n\in \N$, the  \emph{$n^{\textit{th}}$ cocktail party graph} $C_n$ is the graph with vertex set $\{1,\ldots,2n\}$, containing the edge $[i,j]$ if and only if $|i-j|\ne n$.
\end{definition}
 See \cref{fig:turan} for an illustration of $C_3$.  Cocktail party graphs are also known as \emph{octahedral graphs}, among other names; see \cite{weisstein} and references therein. Note that $C_n$ has exactly $2^n$ maximal cliques, formed by choosing one vertex from each of the $n$ independent pairs $\{\{i,i+n\} \mid 1\leq i\leq n\}.$ 

 \begin{figure}[t]
\centering
    \includegraphics[width=0.25\textwidth]{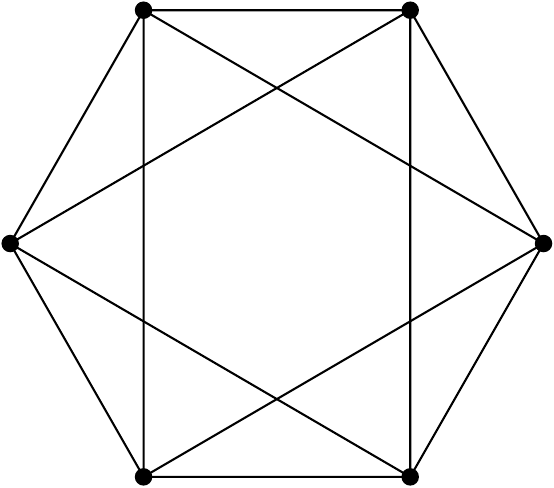}
    \caption{The cocktail party graph $C_3$.}
    \label{fig:turan}
\end{figure}

 \begin{proof}[Proof of \Cref{cor:intro-3}\,(ii)]
 For $n \in \N$, let $X$ denote the vertex set of $C_n$, viewed as a metric space with the metric on $X$ equal to twice the shortest path metric on $C_n$.  Let  $\delta \in [1, \sqrt 2)$ and let 
 \[
 \mathcal F\colon \mathbb N^{\op}\times \T\to \Simp
 \]
 be a functor that is $\delta$-homotopy interleaved with $\srips(X)$.  To establish the result, it suffices to show that $\mathcal F$ has size exponential in $|X|=2n$.
 
 Note that $\dim H_0(\srips(X))_{n,1}=2^n$; indeed, each of the $2^n$ maximal cliques in the neighborhood graph $\mathcal{G}(X)_{1}=C_n$ forms a maximal $(n-1)$-simplex in $\rips(X)_{1}$, giving $2^n$ isolated vertices in $\srips(X)_{n,1}$ by \cref{guacaroni}. Furthermore, the induced map \[H_0(\srips(X))_{n,1} \longrightarrow H_0(\srips(X))_{n,\delta^2}\] is an equality, since $\rips(X)_1 = \rips(X)_{\delta^2}$. Since $\srips(X)$ and $\mathcal F$ are $\delta$-homotopy interleaved, this equality factors through $H_0(\mathcal F)_{n, \delta}$ as
    \begin{center}
\begin{tikzcd}
{H_0(\srips(X))_{n,1}} \arrow[rd] \arrow[rr, "="] &                                        & {H_0(\srips(X))_{n,\delta^2}} \\
                                                   & H_0(\mathcal F)_{n, \delta} \arrow[ru] &                          
\end{tikzcd}
    \end{center}
    Thus  $\dim H_0(\mathcal F)_{n,\delta}\geq 2^n$, implying that $\mathcal F$ has size exponential in $|X|$.
\end{proof}

We conclude this section with two general remarks giving context for our use of cocktail party graphs.

\begin{remark}
Cocktail party graphs are examples of \emph{Tur\'an graphs}, which occur frequently in extremal graph theory \cite{aignerTuranGraph1995}.     In brief, for $n \in \N$ and ${1 \leq r \leq n}$, the \emph{Tur\'an graph} $T(n,r)$ is the complete $r$-partite graph whose partite sets are as close to equal cardinality as possible.  For example, $T(2n,n)$ is the cocktail party graph $C_n$, while $T(n,1)$ is the graph with $n$ vertices and no edges, and $T(n,n)$ is the complete graph $K_n$. The $n$-vertex graph with the largest possible number of maximal cliques is given by the Tur\'an graph $T(n, \lceil n/3 \rceil)$ \cite[Theorem ~$1$]{moonCliquesGraphs1965}.
\end{remark}

\begin{remark}
A class of graphs $\Omega$ is said to have \emph{few cliques} if the number of maximal cliques of any graph $G=(V,E)$ in $\Omega$ is polynomial in $|V|$ \cite{prisnerGraphsFew1995a, rosgenComplexityResults2007}.  
In view of  \Cref{thm:main-theorem-intro}\,(i), to understand when the subdivision bifiltrations of a class of clique filtrations have poly-sized skeleta, one wants to identify conditions under which a class of graphs has few cliques.  This is well studied.  Graph classes known to have few cliques include trees, chordal graphs, planar graphs \cite[Corollary ~$2$]{woodMaximumNumber2007}, graphs with bounded degree \cite[Theorem~3]{woodMaximumNumber2007} or bounded degeneracy \cite[Theorem~3]{eppsteinListing10a}, and intersection graphs of convex polytopes with a bounded number of facets \cite[Theorem~15]{brimkovHomotheticPolygons2018}.  This last class of graphs will be of particular relevance to us in \Cref{sec:Lp-spaces}. 

We have seen above that the class of cocktail party graphs does not have few cliques.  Moreover, it is shown in \cite{farberUpper93} that this observation has the following satisfying converse:  Let $\Omega$ be a class of graphs for which there exists a constant $n_0$ such that for all $n>n_0$, no graph $G\in\Omega$ has a vertex-induced subgraph isomorphic to $C_n$.  Then  $\Omega$ has few cliques.  
\end{remark}

\section{Subdivision-Rips Bifiltrations in $\ell_p$ Space}\label{sec:Lp-spaces}\label{sec:taxicab}

\subsection{Data in $\R^2$}

We next prove \cref{cor:intro-plane}.  First, we recall the statement:

\begin{repcorollary}{cor:intro-plane}\mbox{}
\begin{itemize}
    \item[(i)] For $X$ a finite set of uniformly spaced points on the unit circle, any simplicial model of $\srips(X)$ has exponential size in $|X|$.
    \item[(ii)] Let $X$  be a finite i.i.d.\ sample of the uniform distribution on the unit square.   There exists a constant $C > 0$ such that
 for any simplicial model $\mathcal F(X)$ of $\srips(X)$,    \[
    \operatorname{Pr}[\textup{size of } \mathcal F(X) \textup{ is at least }\exp(C|X|^{1/3})] \rightarrow 1
    \]
    as $|X| \rightarrow \infty$.
\end{itemize}
\end{repcorollary}

Recall from \cref{Sec:Rips_Cech} that we denote the $r$-neighborhood graph of a metric space $X$ as $\mathcal G(X)_r$.  For the proof of \Cref{cor:intro-plane}\,(ii), we will use a result of  Yamaji \cite{yamajiNumberMaximal2023a} about the number of maximal cliques in random geometric graphs:

\begin{proposition}[\cite{yamajiNumberMaximal2023a}, Theorem $1.1$]\label{theo:geo-graph-cliques}
    For fixed $r\in (0,1/2)$, let $X$ be a finite, i.i.d. sample of the uniform distribution on $[0,1]^2$, and let $\kappa_r$ denote the number of maximal cliques in $\mathcal G(X)_r$.  There exists a constant $C>0$ such that 
    \[
    \operatorname{Pr}[\kappa_r \geq \exp(C|X|^{1/3})] \rightarrow 1
    \]
    as $|X| \rightarrow \infty$.\end{proposition}
A corresponding upper bound on $\kappa_r$ is also given in \cite{yamajiNumberMaximal2023a}, but is not needed here.

\begin{proof}[Proof of \Cref{cor:intro-plane}]
To prove (i), we adapt an example from the proof of \cite[Theorem ~$1$]{guptaMaximalCliques2005a}: For $n \in \N$, let $X_n$ be $2n$ points uniformly spaced on the unit circle $S^1 \subset \R_2^2$. Observe that for $\delta > 0$ sufficiently small, the neighborhood graph $\mathcal G(X_n)_{1-\delta}$ is isomorphic to the cocktail party graph $C_n$ (\Cref{Def:Cocktail_Party_Graph}). Thus, $\mathcal G(X_n)_{1-\delta}$ has $2^n$ maximal cliques; see \Cref{fig:circle} for an illustration of the case $n=4$.  The result now follows from \cref{Cor:Thm_Generalization}\,(ii).

Item (ii) is immediate from \cref{Cor:Thm_Generalization}\,(ii) and \Cref{theo:geo-graph-cliques}.
\end{proof}

\begin{figure}[t]
    \centering
    \includegraphics[width=0.25\textwidth]{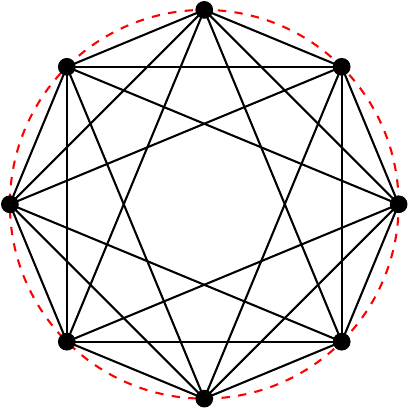}
    \caption{The neighborhood graph $N(X_4)_{1-\delta}$.}
    \label{fig:circle}
\end{figure}

\subsection{Data in $\ell_p$ Space}
\label{Sec:Lp-Approximation}

Next, we prove \cref{cor:intro-2}, which concerns the subdivision bifiltrations of data in $\R^d$ with the $\ell_p$-metric.  Recall the statement:

\begin{repcorollary}{cor:intro-2}
 Fixing $d\in \N$, let $X\subset \R^d_p$.
    \begin{itemize}
        \item [(i)] If $d=1$ or $p\in \{1, \infty\}$, then $\mathcal{NR}(X)$ is a simplicial model of $\srips(X)$ whose $k$-skeleton has size $O(|X|^{\alpha(k+1)+2})$, where 
        \[\alpha=\alpha(d,p)=\begin{cases}
        d &\text{  for  $p=\infty$ or $d=1$ or $(d,p)=(2,1)$},\\

        d2^{d^2-1} &\text{ for $p=1$ and $d>2$}. 
        \end{cases}\]
        \item [(ii)] For any $\eps \in (0,1)$, there exists a $(1+\eps)$-approximation to $\srips(X)$ whose $k$-skeleton has size $O(|X|^{\beta (k+1)+2}),$ where 
        $\beta = c_d \eps ^{-(d^2-1)/2}$ for some constant $c_d$ depending on $d$.  
        \end{itemize}
\end{repcorollary}

For the proof, we will need the following definition:

\begin{definition}[\cite{mckeeTopicsIntersection1999}]
 For a collection of non-empty sets $U$, the \emph{intersection graph of $U$}, denoted $\Gamma_{U}$, is the 1-skeleton of the nerve of $U$.  That is, $\Gamma_{U}$ is the graph with vertex set $U$ and an edge $[A,B]$ if and only if $A \cap B \neq \emptyset$.
\end{definition}

For $p\in [1,\infty]$, $r\in [0,\infty)$, and $x\in \R^d$, let \[B(x)^p_r=\{y\in \R^d\mid \|y-x\|_p\leq r\}.\]
\begin{lemma}\label{Lem:Nbhd_Intersection_Graph}
For $p\in [1,\infty]$ and $X\subset \R^d_p$, we have $\mathcal{G}(X)_r=\Gamma_{U_r}$, where \[U_r=\{B(x)_r^p \mid x \in X\}.\]  
\end{lemma}

\begin{proof}
Consider $x,y\in \R^d$ and let $a=\frac{1}{2}(x+y)$.  Then since $\|\cdot\|_p$ is a norm, \[\|x-a\|_p=\|y-a\|_p=\frac{1}{2}\|x-y\|_p.\] 
By the triangle inequality for the $\ell_p$-metric, there exists no $z\in \R^d$ such that \[\max(\|x-z\|_p,\|y-a\|_p)<\frac{1}{2}\|x-y\|_p.\] Thus, if $r\leq \frac{1}{2}\|x-y\|_p$, then $[x,y]$ is contained in both $\mathcal{G}(X)_r$ and $\Gamma_{U_r}$, while if 
$r> \frac{1}{2}\|x-y\|_p$, then $[x,y]$ is contained in neither $\mathcal{G}(X)_r$ nor  $\Gamma_{U_r}$.
\end{proof}

Our main tool for proving \cref{cor:intro-2} is a bound of Brimkov et al.~\cite{brimkovHomotheticPolygons2018} on the number of maximal cliques in intersection graphs of convex polytopes:

\begin{proposition}[{\cite[Theorem 15]{brimkovHomotheticPolygons2018}}]\label{prop:brimkov}
If $G$ is the intersection graph of $n$ convex polytopes in $\R^d$ whose facets are parallel to at most $k$ hyperplanes, then the number of maximal cliques of $G$ is at most $n^{dk^{d+1}}.$
\end{proposition}

The next result strengthens \cref{prop:brimkov}\,(ii) in the special case that each polytope is a \emph{box}, i.e., the Cartesian product of $d$ closed intervals in $\R$ of positive length. It was stated without proof in \cite{spinradIntersectionContainment2003a} and proven in \cite{lavrovProofNumber}.  
For completeness, we include a proof here.

\begin{lemma}[{\cite{spinradIntersectionContainment2003a,lavrovProofNumber}}]\label{lem:Lavrov}
If $G$ is the intersection graph of a set $V$ of boxes in $\R^d$, then the number of maximal cliques of $G$ is at most $|V|^d$.
\end{lemma}

\begin{proof}
We call $x=(x_1,\ldots,x_d)\in \R^d$ an \emph{anchor} if each $x_i$ is the minimum $i^{\text{th}}$ coordinate of some box in $V$.  

Boxes have the \emph{binary intersection property} (also known as the \emph{$2$-Helly property}): The vertex sets of cliques in $G$ are exactly the simplices in $\Ner(V)$ \cite{danzerIntersectionProperties1982}.  In particular, if $\sigma\subset V$ is the vertex set of a maximal clique in $G$, then there exists $x\in \R^d$ such that $\sigma= \{S\in V\mid x\in S\}$.  In fact, it is readily checked that $x$ may be chosen as an anchor.  Thus, the number of maximal cliques of $G$ is at most the number of anchors, which at most $|V|^d$.  
\end{proof}

\begin{proof}[Proof of \Cref{cor:intro-2}\,(i)]
In what follows, we assume all metric balls are closed.  Suppose $X\subset \R^d_p$.  If $d=1$ or $p=\infty$, then metric balls in $\R^d_p$ are boxes, while if $(d,p)=(2,1)$, then metric balls are rotations of boxes (squares) by 45 degrees.  Therefore, in each of these cases,  \Cref{lem:Lavrov} implies that for each $r\in [0,\infty)$, the number of maximal cliques in $\mathcal G(X)_r$ is at most $|X|^d$.  Recall that maximal cliques in $\mathcal{G}(X)_r$ correspond to maximal simplices in $\mathcal R(X)_r$.  As $\mathcal R(X)$ has at most $O(|X|^2)$ different complexes, it follows that 
 \[m_k(\mathcal R(X))=O\left(|X|^{d(k+1)+2}\right).\]

Next consider the case $p=1$ and $d>2$.  A metric ball in $\R^d_1$ is a \emph{hyperoctahedron}, which is a convex polyhedron whose facets are parallel to $2^{d-1}$ hyperplanes; these hyperplanes are independent of center and radius of the ball.  Thus, by \Cref{prop:brimkov}, the number of maximal cliques in $\mathcal{G}(X)_r$ is at most $|X|^{\alpha}$, where $\alpha= d2^{d^2-1}$.  It follows that \[m_k(\mathcal R(X))=O\left(|X|^{\alpha(k+1)+2}\right).\qedhere\]
\end{proof}

To prove \cref{cor:intro-2}\,(ii), we approximate metric balls in $\R_p^d$ by convex polytopes and appeal to \Cref{prop:brimkov}.  To bound the number of facets of these polytopes, we will require a bound on the maximum size of a packing of metric balls on a hypersphere.  We discuss this first.

\begin{definition}\label{Def:Packing}
For $Y$ a metric space and $\delta,\delta'>0$, define a $(\delta,\delta')$-\emph{packing of $Y$} to be a set $W\subset Y$ such that
 \begin{enumerate}
 \item $\min_{w\in W} \|w-y\|_2\leq \delta$ for each $y\in Y$, and 
 \item $\|w-w'\|_2\geq \delta'$ for any distinct points $w,w'\in W$.  
 \end{enumerate}
Define a \emph{$\delta$-packing of $Y$} to be a $(\delta,\delta)$-packing.  
\end{definition}

For $d\in \N$ and $\delta>0$, let $\pack(\delta)$ denote the maximum size of a $\delta$-packing of the unit sphere $S^{d-1}\subset \R^d$, where $S^{d-1}$ is given the extrinsic metric.  Several bounds on $\pack(\delta)$ are known; see \cite{4247333} and \cite[Section 6.3]{boroczky2004finite}. 
In particular, a simple volume computation given in \cite{4247333} shows that 
\begin{equation}\label{Eq:SonthaliaPacking}
\pack(\delta)\leq 2d(1+2/\delta)^{d-1}.
\end{equation}

We will need the following notation for the statement of \Cref{thm:polytope-hdorff} below.  For $d\in \N$ and $\lambda>0$, let 
 \[f(\lambda)=\pack\left(\frac{1}{3\sqrt{\lambda}}\right).\]  
By \eqref{Eq:SonthaliaPacking}, we have
\begin{equation}\label{Eq:Sonthalia_Impliciation}
f(\lambda)\leq 2d(1+6\sqrt{\lambda})^{(d-1)}.
\end{equation}
In particular, $f(\lambda)=O(\lambda^{(d-1)/2})$ for fixed $d$.  

For $X\subset \R^d$, let
    \[
    \mathcal O(X)^p_r = \bigcup_{x \in X} B(x)^p_r,
    \]  
and let $z$ denote the origin in $\R^d$. 

\begin{proposition}[\cite{dudleyMetricEntropy1974a, har-peledProof19}]\label{thm:polytope-hdorff}
For $d\geq 2$, suppose $A\subset B(z)^2_1$ is closed, convex and non-empty.  Then for any $\eps \in (0,1)$, there exists a convex polytope $P$ with at most $f(1/\epsilon)$ facets such that 
\begin{equation}\label{eq:Polytop_Approx}
 A \subset P \subset \mathcal O(A)^2_\eps.
 \end{equation}
\end{proposition}

\cref{thm:polytope-hdorff} is due to Dudley \cite{dudleyMetricEntropy1974a}.  Proofs also appear in work of Har-Peled and Jones  \cite{har-peledProof19} and Reisner et al.  \cite[Proposition 2.7]{reisner2001dropping}.    Our statement of the result follows \cite{har-peledProof19}.  We now outline the proof given in \cite{har-peledProof19}, in order to elucidate the role of hypersphere packing.

\begin{proof}[Outline of proof of \cref{thm:polytope-hdorff}]
Let $S\subset \R^{d}$ be the $(d-1)$-dimensional sphere of radius $3$ centered at the origin, and let $W$ be a $\sqrt{\epsilon}$-packing of $S$.  To construct $P$, for each $w\in W$ let $n(w)\in \partial A$ be the nearest neighbor of $w$ in $A$, and let $H_w$ be the halfspace containing $A$ whose boundary contains $n(w)$ and is orthogonal to $w-n(w)$.  Let $P=\bigcap_{w\in W} H_w$.  To obtain the bound on the number of facets of $P$, we note that $P$ has at most $|W|$ facets and that $|W|\leq f(1/\epsilon)$.  For the argument that $P$ satisfies \cref{eq:Polytop_Approx}, see \cite{har-peledProof19}. 
\end{proof}

We will use \cref{thm:polytope-hdorff} to prove the following:

\begin{proposition}\label{Prop:Lp_Ball_Approx}
     For all $d\geq 2$, $p \in [1, \infty]$, and $\eps \in (0,1)$, there exists a polytope $P$ with at most $f(\sqrt{d}/\epsilon)$ facets satisfying 
     \[
     B(z)^p_1 \subset P \subset B(z)^p_{1+\eps}.
     \]
\end{proposition}

\begin{proof}
 Recall  that for any $x\in \R^d$ and $1 \leq q \leq q'\leq \infty$, we have 
    \[\lVert x \rVert_{q'} \leq \lVert x \rVert_q \leq d^{\frac{1}{q}-\frac{1}{q'}} \lVert x\rVert_{q'},\]
    where $1/\infty=0$.
    In the case $p \in [2,\infty]$, applying the second inequality with $q=2$ and $q'=p$ yields $
    B(z)^p_{1/d'} \subset B(z)^2_1$, where $d'=d^{\frac{1}{2}-\frac{1}{p}}$.   Note that $d'\leq \sqrt{d}$.  Let $\eps'=\eps/d'$.  Taking $A=B(z)^p_{1/d'}$ in \Cref{thm:polytope-hdorff} implies that there is a polytope $P'$ with at most $f(1/\eps')= f(d'/\eps)$ facets satisfying 
    \[
    B(z)^p_{1/d'} \subset P' \subset \mathcal O(B(z)^p_{1/d'})^2_{\eps'} \subset  \mathcal O(B(z)^p_{1/d'})^p_{\eps'} =B(z)^p_{(1+\eps)/d'}.
    \]
    Linearly rescaling $P'$ by a factor of $d'$, we obtain a polytope $P$ with the same number of facets satisfying \[B(z)^p_1 \subset P \subset B(z)^p_{1+\eps}.\]
   Next, consider the case $p \in [1,2)$.  Note that $
    B(z)^p_1 \subset B(z)^2_1.$  Let $d'=d^{\frac{1}{p}-\frac{1}{2}}$ and $\eps'=\eps/d'$.  As above, we have $d'\leq \sqrt{d}$.  Taking $A=B(z)^p_1$ in \Cref{thm:polytope-hdorff} implies that there exists a polytope $P$ with at most $f(1/\eps')=f(d'/\epsilon)$ facets satisfying 
    \[
    B(z)^p_1 \subset P \subset \mathcal O(B(z)_1^p)^2_{\eps'} \subset \mathcal O(B(z)_1^p)^p_{\eps} =B(z)^p_{1+\eps}.
  \qedhere  \]
\end{proof}

We can now give the construction that we will use to prove \Cref{cor:intro-2}\,(ii). Given a finite set $X \subset \R^d$ and $\epsilon \in (0,1)$, \Cref{Prop:Lp_Ball_Approx} implies that there exists a polytope $P$ with $f(\sqrt{d}/\epsilon)$ facets satisfying \[B(z)^p_1 \subset P\subset B(z)^p_{1+\eps}.\]  For any $x\in X$ and $r \in\T$,  linearly rescaling $P$ by a factor of $r$ and then translating by $x$ gives a polytope $P(x)_r$ satisfying 
\[
B(x)^p_r \subset P(x)_r \subset B(x)^p_{r(1+\eps)}.
\]
Note that if $r>0$, then $P(x)_r$ has the same number of facets as $P$, while if $r=0$, then $P(x)_r=\{x\}$.
Let $V_r := \{ P(x)_r \mid x \in X\}$  and $\mathcal P (X)_r := \Gamma_{V_r}$. \Cref{Fig:Spherical-vs-Polyhedral-Graphs} shows an example in the plane. These graphs assemble into a graph filtration
\[
\mathcal P(X):\T\to \Simp.
\]

\begin{figure}[h]
  \centering
  \begin{subfigure}[c]{0.45\textwidth}
    \centering
    \includegraphics[width=\textwidth]{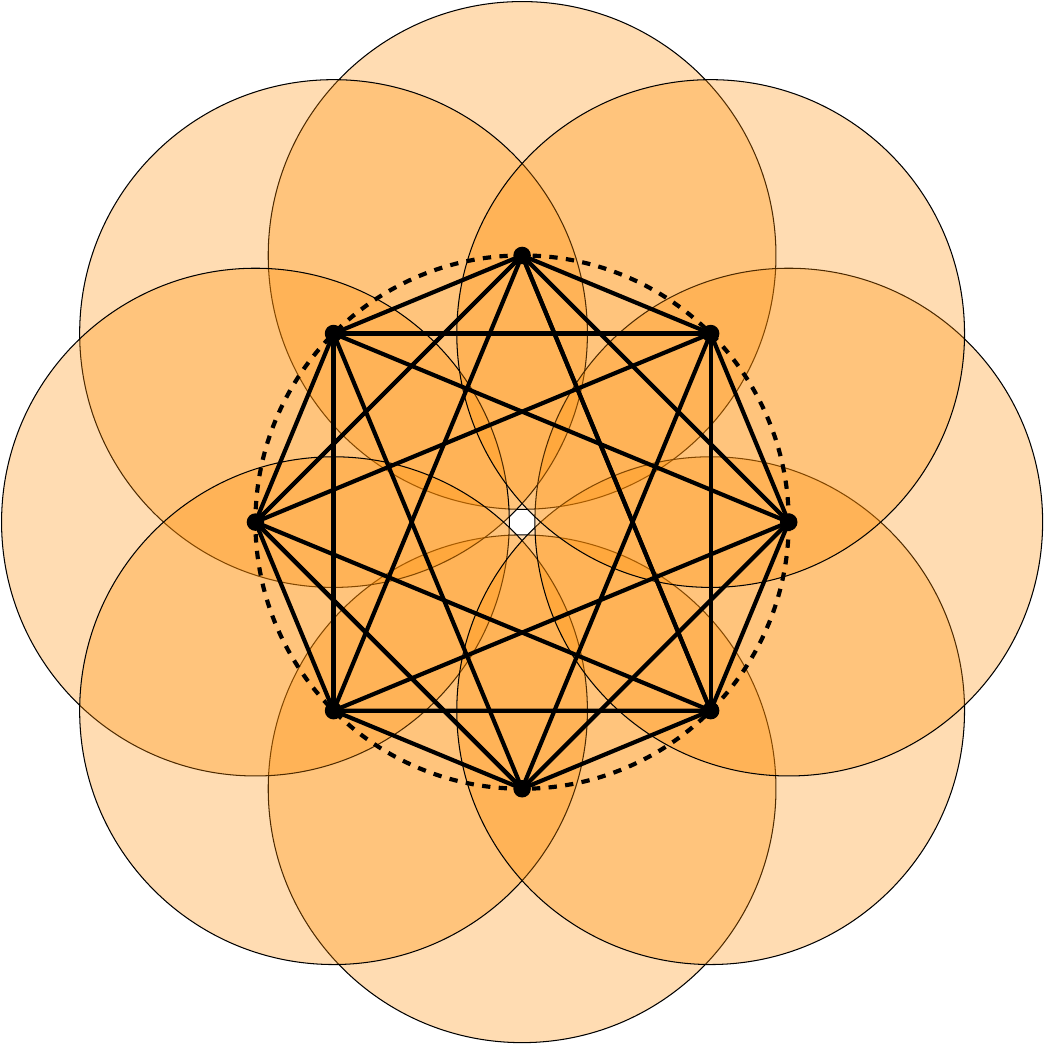}
    \caption{The graph $C_4$ realized as an intersection graph of disks.\\ {}}
  \end{subfigure}
    \hfill
  \begin{subfigure}[c]{0.45\textwidth}
    \centering
    \includegraphics[width=\textwidth]{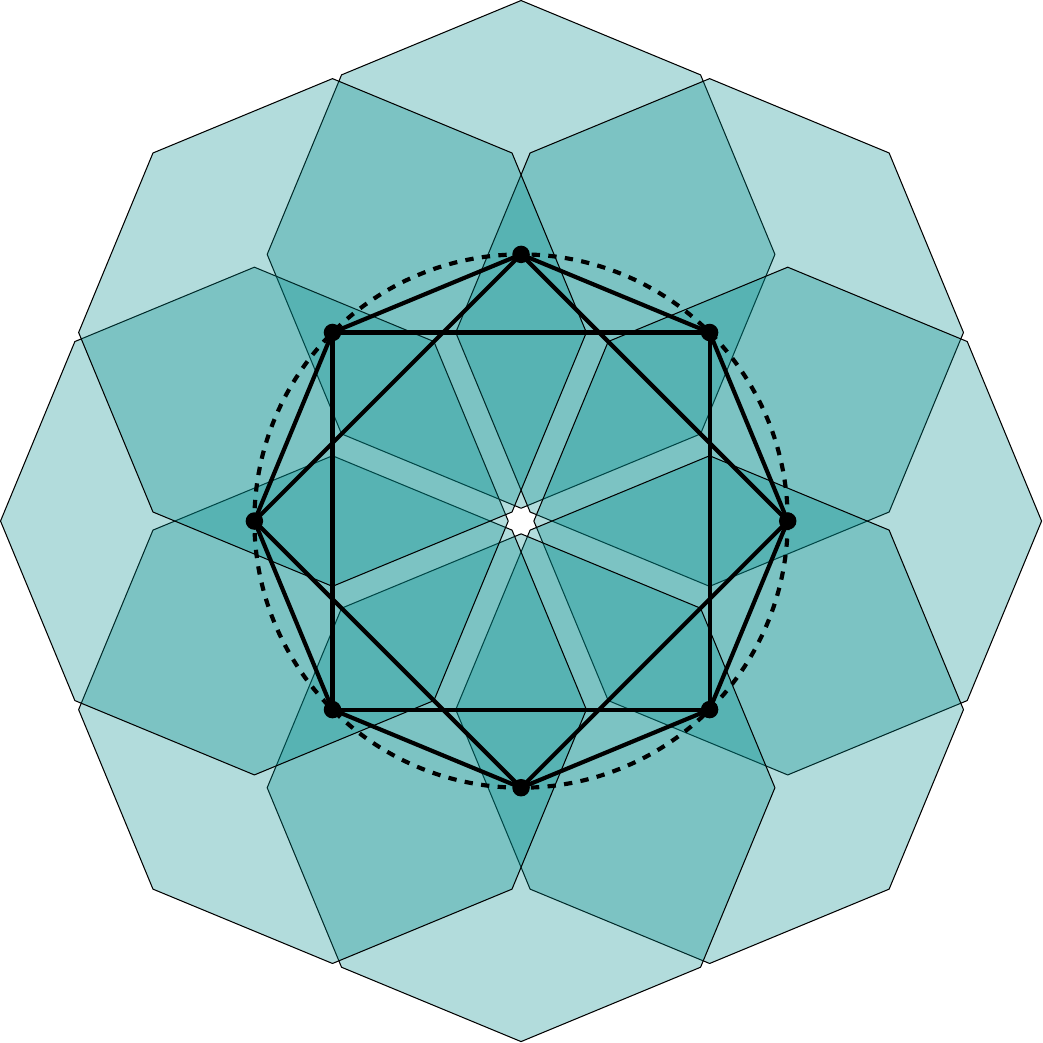}
    \caption{The intersection graph obtained by replacing disks with inscribed octagons.}
  \end{subfigure}
  \caption{Spherical vs polyhedral neighborhood graphs.}
  \label{Fig:Spherical-vs-Polyhedral-Graphs}
\end{figure}

\begin{lemma}\label{lem:poly-rips-interleaved}
   The clique filtrations  $\rips(X)$ and $\overline{\mathcal P}(X)$ are $(1+\eps)$-interleaved.
\end{lemma}

\begin{proof}
    By construction, for each $x\in X$ and $r \in\T$, we have 
    \[
    B(x)^p_r \subset P(x)_r \subset B(x)^p_{r(1+\eps)},
    \]
    so by \Cref{Lem:Nbhd_Intersection_Graph}, we have
    \[
    \mathcal{G}(X)_r \subset \mathcal P(X)_r \subset \mathcal{G}(X)_{r(1+\eps)}.
    \]
    Passing to clique complexes then yields
    \[
    \rips(X)_r \subset \overline{\mathcal P}(X)_r \subset \rips(X)_{r(1+\eps)},
    \]
which implies the result.  
\end{proof}

The following is an immediate consequence of \Cref{lem:poly-rips-interleaved} and \Cref{Prop:Inerleaving_and_Subdivision}:

\begin{lemma}\label{lem:sp-sr-interleaved}
    The bifiltrations $ \mathcal {S}\overline{\mathcal P}(X)$ and $\srips(X)$ are $(1+\eps)$-interleaved.
\end{lemma}

\begin{proof}[Proof of \Cref{cor:intro-2} (ii)]
The case $d=1$ is covered by \Cref{cor:intro-2}\,(i), so assume that $d\geq 2$.  Each $\mathcal P(X)_r$ is an intersection graph of $|X|$ polytopes whose facets are parallel to at most $f(\sqrt{d}/\epsilon)$ distinct hyperplanes.
Inequality \eqref{Eq:Sonthalia_Impliciation} implies that 
\[f(\sqrt{d}/\epsilon)\leq 2d\left(1+6\sqrt[4]{d}/\sqrt{\epsilon}\right)^{(d-1)}\leq  2\cdot 7^{d-1}d^{(\frac{d-1}{4}+1)}(1/\epsilon)^{\frac{d-1}{2}}.\]  Thus, letting
\begin{equation}\label{eq:1.8constant}
c_d=d\left( 2\cdot 7^{d-1}d^{(\frac{d-1}{4}+1)}\right)^{d+1},
\end{equation}
\Cref{prop:brimkov} implies that the number of maximal cliques in $\mathcal P(X)_r$ is $O(|X|^{\beta})$, where
$\beta = c_d \eps ^{-(d^2-1)/2}$. 
Therefore, \[m_k\left(\overline{\mathcal P}(X)\right)=O\left(|X|^{\beta(k+1)+2}\right).\]  
The result now follows from \cref{Cor:Thm_Generalization}\,(i) and \Cref{lem:sp-sr-interleaved}.
\end{proof}

\section{Computation}\label{Sec:Computation}
We now briefly discuss computation. Specifically, we address the problem of computing the fixed-dimensional skeleta of the nerve-based semifiltrations appearing in the statements of our main theorems.  We concern ourselves only with the modest goal of computing these skeleta in time polynomial in the total size of input and output (i.e., \emph{output-sensitive polynomial time}), leaving a more nuanced study of computation to future work. 

Observe that for a finitely presented filtration $\mathcal F\colon \N\to \Simp$, to compute the $0$-skeleton of $\mathcal {NF}$ it suffices to compute the set $\Sigma_t$ of maximal simplices of each $\mathcal F_t$.  Once we have done so, to compute the $k$-skeleton, it suffices to test all $j$-sets of each $\Sigma_t$ for non-empty intersections, where $j\leq k+1$.  For fixed $k$, this can be done naively in polynomial time.

When $\mathcal F$ is the clique filtration of a graph filtration $\mathcal G$ (e.g., a Rips filtration), we can compute the maximal simplices of $\mathcal F_t$ in output-sensitive polynomial time using well-known algorithms for enumerating all maximal cliques in a graph \cite{tsukiyamaNewAlgorithm1977,makino2004new,cazals2008note}.   

To compute the 0-skeleton of the semifiltration $\mathcal{J}(X)$ of \cref{cor:intro-3}, it suffices to compute the maximal simplices of the intrinsic \v Cech filtration $\mathcal I(X)$.  As these maximal simplices are metric balls (see the proof of \cref{lem:cech-simplices}), we can compute them naively in polynomial time by computing all metric balls at each index and doing containment tests to determine which balls are maximal.

To compute a semifiltration as in the statement of \Cref{cor:intro-2}\,(ii) in polynomial time (for fixed $d$ and $\epsilon$), we adapt the construction underlying the proof of  \Cref{cor:intro-2}\,(ii).   The main task is to compute the filtered intersection graph $\overline{\mathcal P}(X)$.  We preface our discussion of this with a technical caveat: Whereas each vertex of $\overline{\mathcal P}(X)_r$ is a polytope with at most $f(\sqrt{d}/{\epsilon})$ facets, our algorithm will produce a variant of $\overline{\mathcal P}(X)_r$ for which this number of facets is most $f(4\sqrt{d}/\epsilon)$.  Hence, the semifiltration output by our algorithm will satisfy the size bound of \Cref{cor:intro-2}\,(ii) for a larger value of the constant $c_d$ than specified in the proof of \Cref{cor:intro-2}\,(ii).   

 To compute $\overline{\mathcal P}(X)$, it suffices to give algorithms for the following two problems, where we assume that each polytope $Q$ is represented as a list of half-spaces whose intersection is $Q$:
\begin{enumerate}
\item Given $\epsilon>0$ and an $\ell_p$-ball $A\subset \R^d$ of radius $3$ centered at the origin, compute a convex polytope $P$ of size at most $f(4/\epsilon)$ satisfying 
the containment condition \eqref{eq:Polytop_Approx} of \cref{thm:polytope-hdorff}.
\item Determine whether two convex polytopes in $\R^n$ have non-empty intersection.  
\end{enumerate}

The second problem can be solved in polynomial time using standard linear programming algorithms \cite[Chapter 7]{matouvsek2007understanding}.  For the first problem, we adapt the construction from the proof of \cref{thm:polytope-hdorff} given in \cite{har-peledProof19}, which we have outlined in \cref{Sec:Lp-Approximation}.  

Letting $\delta=\sqrt{\epsilon}$, recall that the construction begins with a $\delta$-packing  $W$ of a sphere $S\subset \R^d$ of radius $3$ centered at the origin.    Given $W$, the construction of $P$ is transparently algorithmic in the present case, as for each $w\in W$, its nearest neighbor in $A$ is the point $3w/\|w\|_p$.  
 
We do not know a deterministic algorithm for computing a $\delta$-packing of $S$.  However, the argument of \cite{har-peledProof19} extends immediately to show that if one instead begins with a $(\delta,\delta/2)$-packing of $S$ (see \cref{Def:Packing}), then the same construction yields a polytope $P$ with at most $f(4/\epsilon)$ facets satisfying \eqref{eq:Polytop_Approx}.  

For fixed $d$ and $\eps$, we can compute a $(\delta,\delta/2)$-packing of $S$ naively in constant time, e.g.,  as follows: First, compute a $(\delta/4)$-sample $W''$ of a hypercube containing $S$, so that $W''$ excludes the origin. For example, $W''$ can be chosen as a uniform grid.  Projecting $W''$ onto $S$ via the normalization map $x \mapsto x/\lVert x \rVert_2$ yields a $(\delta/2)$-sample $W'=\{w_1,\ldots,w_l\}$ of $S$.  Finally, compute the set $W=W_l$ defined inductively by $W_1=\{w_1\}$, and for $i\in \{2,\dots,l\}$, 
\[W_i=\begin{cases}
W_{i-1} \cup \{w_i\} &\textup{ if }  \| w-w_i\| \geq \delta/2 \textup{ for all } w\in W_{i-1},\\
W_{i-1}   &\textup{ otherwise}. 
\end{cases}\]
It is clear that $W$ is a $(\delta,\delta/2)$-packing of $S$.  
\appendix

\section{Constructing a Bifiltration from a Semifiltration}
\label{Semi_Filtration_to_Bifilration}
The simplicial models in this paper are generally not bifiltrations, but rather semifiltrations.  However, computation of 2-parameter persistent homology is simplest and most closely aligned with existing software \cite{fugacci2023compression,lesnickComputingMinimal2022a,lesnickInteractiveVisualization2015,bauer2023efficient} when the input is bifiltered.  Therefore, we wish to understand how to efficiently convert a semifiltration into a weakly equivalent bifiltration.  

Extending work of Dey et al. \cite{dey2014computing}, Kerber and Schreiber \cite{kerber2019barcodes} have given a nice solution to the 1-parameter version of this problem.  Namely, \cite{kerber2019barcodes} gives a construction to transform a finitely presented functor $\mathcal F\colon \mathbb N \to \Simp$ into a weakly equivalent filtration $\mathcal F'$ such that, letting $|\cdot|$ denote size, we have
\begin{equation}\label{eq:size_tower_to_filtration}
|\mathcal F'|=O(\Delta|\mathcal F|\log |\mathcal F^{\, 0}|),
\end{equation}
where $\Delta=\max_{t} \dim(\mathcal F_t)$ and $\mathcal F^{\, 0}$ is the 0-skeleton of $\mathcal F$. The authors also give an efficient algorithm for computing $\mathcal F'$.  

The approach of \cite{kerber2019barcodes} extends to convert a finitely presented simplicial semifiltration $\mathcal A\colon \mathbb N^{\mathrm{\op}}\times \mathbb N \to \Simp$ to a weakly equivalent bifiltration $\mathcal A'$, with the same asymptotic increase in size.  This amounts to applying the construction of \cite{kerber2019barcodes} to the functor $\mathcal{A}_{1,-}\colon \mathbb N \to \Simp$, yielding a filtration $\mathcal{A'}_{1,-}$ that inherits a bifiltered structure from $\mathcal A$.  

To explain this more fully, we recall some details of the construction of \cite{kerber2019barcodes}.  By factoring the structure maps of  $\mathcal F\colon \mathbb N \to \Simp$, we may assume without loss of generality that each structure map $\mathcal F_{t\subto t+1}$ is either is an identity map, an inclusion inserting a single simplex, or a \emph{vertex collapse}, i.e., a surjection that maps exactly two vertices $\{u,v\}$ to a single vertex $w$.  If we choose the factorization of each structure map to have a minimal number of factors, then the factorization does not change the size of $\mathcal F$. 

The construction of $\mathcal F'$ is inductive in $t$. We maintain a label for each vertex in $\mathcal F'_t$, marking the vertex as \emph{active} or \emph{inactive}.  If $\mathcal F_{t\subto t+1}$ is an identity map, then we also take $\mathcal F'_{t\subto t+1}$ to be an identity.  If $\mathcal F_{t\subto t+1}$ inserts a new simplex $\sigma$, then we take $\mathcal F'_{t+1}=\mathcal F'_{t}\cup \{\sigma \}$, and if $\sigma$ is a vertex, then we also mark it as active in $\mathcal F'_{t+1}$.  Now suppose $\mathcal F_{t\subto t+1}$ is vertex collapse mapping $\{u,v\}$ to $w$.  To construct $\mathcal F'_{t+1}$, we choose one of the vertices $\{u,v\}$.  Let $V$ be the subcomplex of the closed star\footnote{The closed star of a vertex $x$ in a simplicial complex is the subcomplex consisting of the cofaces of $x$ along with their faces.} of $v$ in $\mathcal F'_{t}$ spanned by active vertices, and let $X[t,{\mathcal F}]$ be the simplicial cone on $V$ with apex $u$.  We take $\mathcal F'_{t+1}=\mathcal F'_{t}\cup X[t,{\mathcal F}]$.  The vertex $v$ is then marked as inactive, and each subsequent instance of $w$ in $\mathcal F$ is renamed $u$.  See \cite[Figure 1]{kerber2019barcodes} for an illustration.   Either choice of $u$ or $v$ leads to a topologically correct construction, so \cite{kerber2019barcodes} makes the choice in a way that minimizes the size of the resulting filtration.  However, the asymptotic bound \eqref{eq:size_tower_to_filtration} on $|\mathcal F'|$ requires the choice of active vertex to be the one that leads to the smaller choice of $F'_{t+1}$.  
Henceforth, we assume without loss of generality that every vertex collapse we encounter is of the form $\{u,v\}\mapsto u$, where $v$ is subsequently marked inactive.
 
Next, we consider how to extend this approach to convert a finitely presented simplicial semifiltration $\mathcal{A}$ to a bifiltration $\mathcal{A'}$.  Without loss of generality, we assume as above that the structure maps $\mathcal{A}_{(1,t)\subto (1,t+1)}$ are either identities, inclusions of single simplices, or vertex collapses.  We apply the 1-parameter construction to $\mathcal{A}_{1,-}$, yielding a filtration $\mathcal{A'}_{1,-}$. Then, to define $\mathcal A'$, for $(k,t)\in \mathbb N^{\mathrm{\op}}\times \mathbb N$, we take $\mathcal{A'}_{k,t}$ to be the maximal subcomplex of $\mathcal{A'}_{1,t}$ on the vertices $\bigcup_{t'\leq t} \mathcal{A}^0_{k,t'}$, where $\mathcal{A}^0$ denotes the 0-skeleton of $\mathcal A$.  Informally, then, $\mathcal{A'}_{k,-}$ is the filtration obtained by restricting of construction of $A'_{1,-}$ to the subfunctor $A_{k,-}\subset A_{1,-}$.

\begin{proposition}
The bifiltration $\mathcal A'$ is weakly equivalent to $\mathcal A$.
\end{proposition}

\begin{proof}
We have a natural simplicial epimorphism $\rho\colon \mathcal A'\to \mathcal A$ defined as follows: For $(k,t)\in \mathbb N^{\op}\times \mathbb N$, a vertex $x\in \mathcal A'_{k,t}$, and $s\leq t$ such that $x\in \mathcal A_{k,s}$, we let $\rho_{k,t}(x)=\mathcal A_{(k,s)\subto (k,t)}(x)$.  One easily checks that this definition does not depend on the choice of $s$.  
For fixed $(k,t)$, let $S\subset \mathbb N$ denote the set of indices $s<t$ such that $A_{(k,s)\subto (k,s+1)}$ is a collapse.   The map $\rho_{k,t}$ factors as a sequence of collapses of the simplicial cones $\{X[s,\mathcal{A}_{k,-}]\mid s\in S\},$ where the collapses are done in order of increasing $s$.    Specifically, if $A_{(k,s)\subto (k,s+1)}$ is the collapse  $\{u,v\}\mapsto u$, then we collapse the cone $X[s,\mathcal{A}_{k,-}]$ onto the image of the simplicial endomorphism of $X[s,\mathcal{A}_{k,-}]$ mapping $v$ to $u$ and every other vertex to itself.  By Quillen's Theorem A for simplicial complexes \cite{barmak2011quillen}, or by a simple discrete Morse theory argument, each cone collapse is a homotopy equivalence, and so the map $\rho_{k,t}$ is a homotopy equivalence.  Hence $\rho$ is an objectwise homotopy equivalence.  
\end{proof}

Finally, we check the following, where as above, $|\cdot|$ denotes size:
\begin{proposition}
Letting $\Delta=\max_{t} \dim(\mathcal A_{1,t})$, we have \[|\mathcal A'|=O(\Delta|\mathcal A| \log |\mathcal A^0|).\] 
\end{proposition}

\begin{proof}   
For any poset $P$, finitely presented functor $\mathcal F\colon P\to \Simp$, and $i\geq 0$, $\Ch{i} \mathcal F $ has a canonical minimal generating set $Z_{\mathcal F}$, we which may identify with a subset of $\sqcup_{p\in P} \mathcal F_{p}$, namely, 
\[Z_{\mathcal F}= \bigsqcup_{p\in P}\{\sigma \in \mathcal F_{p} \mid \sigma\not\in \im \mathcal F_{p'\subto p} \textup{ for any }p'<p\}.\]

By the definition of $\mathcal A'$, we have natural injections $j\colon Z_{\mathcal A}\hookrightarrow Z_{\mathcal A'}$
and $Z_{\mathcal A'}\setminus \im(j)\hookrightarrow Z_{\mathcal A'_{1,-}}$.  
 It follows that 
\begin{align*}
\sum_{i} \beta_0(\Ch{i}\mathcal A')
&\leq  \sum_{i} \beta_0(\Ch{i}\mathcal A)+\sum_{i} \beta_0(\Ch{i}(\mathcal A'_{1,-}))\\
&\leq |\mathcal A|+|\mathcal A'_{1,-}|\\
&=O(\Delta|\mathcal A| \log |\mathcal A^0|),
\end{align*}
where we have applied the 1-parameter size bound \eqref{eq:size_tower_to_filtration} to $A'_{1,-}$  in the last equality.  The result now follows from \cref{Prop:Betti_Numbers_Semi_Filtrations}.  
\end{proof}

\subsection*{Acknowledgements}
We thank Sariel Har-Peled and Shlomo Reisner for helpful discussions about Dudley's theorem (\cref{thm:polytope-hdorff}).  

\bibliographystyle{plainurl}
\bibliography{bibliography}

\end{document}

%% file: Preamble.tex
\usepackage{amssymb, amsfonts, amsthm, amsmath, mathrsfs, todonotes, tikz-cd, easytable,lineno}
\usepackage[linktocpage=true,hidelinks]{hyperref}
\usepackage{caption}
\usepackage{graphicx}
\usepackage{subcaption}
\usepackage[capitalise]{cleveref}
\usepackage{comment}
\usepackage[T1]{fontenc}
\setuptodonotes{inline}

\usepackage{trimclip}
\makeatletter
\DeclareRobustCommand{\subto}{%
  \mathrel{\mathpalette\short@to\relax}%
}

\newcommand{\short@to}[2]{%
  \mkern2mu
  \clipbox{{.35\width} 0 0 0}{$\m@th#1\vphantom{+}{\rightarrow}$}%
  }
\makeatother

\usepackage{mathtools}
\usepackage{extarrows}
\usepackage[numbers,sort&compress]{natbib}

\newtheorem*{rep@theorem}{\rep@title}
\newcommand{\newreptheorem}[2]{%
\newenvironment{rep#1}[1]{%
 \def\rep@title{#2 \ref{##1}}%
 \begin{rep@theorem}}%
 {\end{rep@theorem}}}
\makeatother

\newtheorem*{rep@corollary}{\rep@title}
\newcommand{\newrepcorollary}[2]{%
\newenvironment{rep#1}[1]{%
 \def\rep@title{#2 \ref{##1}}%
 \begin{rep@corollary}}%
 {\end{rep@corollary}}}
\makeatother

\numberwithin{equation}{section}
\newtheorem{theorem}{Theorem}[section]
\newreptheorem{theorem}{Theorem}

\newtheorem{lemma}[theorem]{Lemma}
\newtheorem{proposition}[theorem]{Proposition}
\newtheorem{corollary}[theorem]{Corollary}
\newreptheorem{corollary}{Corollary}
\newtheorem{conjecture}[theorem]{Conjecture}

\theoremstyle{definition}
\newtheorem{definition}[theorem]{Definition}
\newtheorem{remark}[theorem]{Remark}
\newtheorem{example}[theorem]{Example}
\newtheorem{notation}[theorem]{Notation}
\newreptheorem{notation}{Notation}

\newcommand{\R}{\mathbb R}

\newcommand{\Z}{\mathbb Z}

\newcommand{\N}{\mathbb N}
\newcommand{\cat}[1]{#1}

\newcommand{\rips}{\mathcal{R}}
\newcommand{\srips}{\mathcal{SR}}
\newcommand{\scech}{\mathcal{S}\check{\mathcal{C}}}

\newcommand{\rhomb}{\mathcal{T}}

\newcommand{\cech}{\check{\mathcal{C}}}

\DeclareMathOperator{\op}{op}

\DeclareMathOperator{\Id}{Id}
\DeclareMathOperator{\im}{im}
\newcommand{\gr}[1]{\text{gr}(#1)}
\newcommand{\birth}{\mathbf b}

\newcommand{\Bary}[1]{{#1}^{+}}
\newcommand{\Img}{\mathrm{im}\,}

\newcommand{\Poset}{{\mathbb N}^{\mathrm{op}}\times [0, \infty)}

\newcommand{\Simp}{\mathbf{Simp}}
\newcommand{\Top}{\mathbf{Top}}

\renewcommand{\Vec}{\mathbf{Vec}}

\newcommand{\eps}{\ensuremath{\epsilon}}

\newcommand{\Ch}[1]{C_{#1}\mkern 2mu}
\newcommand{\Ho}[1]{H_{#1}\mkern 2mu}

\DeclareMathOperator{\Ner}{Nrv}

\DeclareMathOperator{\Nullity}{Nullity}
\DeclareMathOperator{\pack}{pack}

\newcommand{\T}{[0,\infty)}

\newcommand{\subdiv}{\mathcal{SF}}
\newcommand{\cCov}{\mathbf{Cov}}
\newcommand{\icech}{\mathcal{I}}

\DeclareMathOperator{\colim}{colim}

\tikzset{%
  symbol/.style={
    draw=none,
    every to/.append style={
      edge node={node [sloped, allow upside down, auto=false]{$#1$}}
    },
  },
}